\pdfoutput=1
\documentclass[12pt,reqno]{amsart}

\usepackage{tikz}
\usetikzlibrary{arrows,backgrounds,calc}
\usepackage{calrsfs}
\usepackage[charter]{mathdesign}
\usepackage{amsthm, amsmath, amsfonts}
\usepackage{graphicx}
\usepackage{hyperref}
\usepackage{lineno}
\usepackage{enumerate}
\usepackage{mathtools}
\usepackage{pgfplots}
\pgfplotsset{compat=1.5}

\allowdisplaybreaks

\newtheorem{theorem}{Theorem}
\newtheorem{proposition}{Proposition}[section]
\newtheorem{lemma}[proposition]{Lemma}
\newtheorem{corollary}[proposition]{Corollary}

\theoremstyle{remark}
\newtheorem*{remark}{Remark}

\theoremstyle{definition}
\newtheorem*{definition}{Definition}

\renewcommand{\MR}[1]{}

\newif\ifdetails
\newcommand{\TODO}[1]%
{\par\fbox{\begin{minipage}{0.9\linewidth}\textbf{TODO:} #1\end{minipage}}\par}
\newcommand{\DLMF}[2]{\cite[\href{http://dlmf.nist.gov/#1.E#2}{#1.#2}]{NIST:DLMF:v1.0.13}}
\newcommand{\details}[1]%
{\ifdetails\par\fbox{\begin{minipage}{0.9\linewidth}\textit{Detail:}
      #1\end{minipage}}\par\fi}

\let\Re\relax

\let\Im\relax

\newcommand{\C}[0]{\mathbb{C}}
\newcommand{\R}[0]{\mathbb{R}}
\newcommand{\Z}[0]{\mathbb{Z}}
\renewcommand{\P}[0]{\mathbb{P}}
\newcommand{\E}[0]{\mathbb{E}}
\newcommand{\V}[0]{\mathbb{V}}
\newcommand{\N}[0]{\mathbb{N}}
\DeclareMathOperator{\Re}{Re}
\DeclareMathOperator{\Im}{Im}
\DeclareMathOperator{\Res}{Res}

\DeclarePairedDelimiter{\abs}{\lvert}{\rvert}
\DeclarePairedDelimiter{\iverson}{[}{]}
\newcommand{\convdistr}{\overset{d}{\longrightarrow}}

\newcommand{\innernode}{\tikz{\node[draw, circle, fill, inner sep=2.5pt] {};}}
\newcommand{\filledcirc}{\tikz{\node[draw, circle, fill, inner sep=1.5pt] {};}}

\title{Fringe Analysis of Plane Trees Related to Cutting and Pruning}

\author[B.~Hackl]{Benjamin Hackl}
\author[C.~Heuberger]{Clemens Heuberger}

\address[Benjamin Hackl, Clemens Heuberger]{Institut f\"ur Mathematik,
  Alpen-Adria-Uni\-ver\-si\-t\"at Klagenfurt, Universit\"atsstra\ss e
  65--67, 9020 Klagenfurt, Austria}
\email{\href{mailto:benjamin.hackl@aau.at}{benjamin.hackl@aau.at}}
\email{\href{mailto:clemens.heuberger@aau.at}{clemens.heuberger@aau.at}}
\thanks{B.~Hackl and C.~Heuberger are supported by the Austrian
  Science Fund (FWF): P~24644-N26 and by the Karl Popper Kolleg
  ``Modeling-Simulation-Optimization'' funded by the Alpen-Adria-Universit\"at Klagenfurt
  and by the Carinthian Economic Promotion Fund (KWF)}

\author[S.~Kropf]{Sara Kropf}
\address[Sara Kropf]{Institute of Statistical Science,
  Academia Sinica, 115 Taipei, Taiwan}
\email{\href{mailto:sarakropf@stat.sinica.edu.tw}{sarakropf@stat.sinica.edu.tw}}

\author[H.~Prodinger]{Helmut Prodinger}
\address[Helmut Prodinger]{Department of Mathematical
  Sciences, Stellenbosch University, 7602 Stellenbosch,
 South Africa}
\email{\href{mailto:hproding@sun.ac.za}{hproding@sun.ac.za}}
\thanks{H.~Prodinger is supported by an incentive grant of the National Research
  Foundation of South Africa. Part of this author's work was done while he visited
  Academia Sinica. He thanks the Institute of Statistical Science for its hospitality.}

\thanks{This is the full version of the extended abstract~\cite{Hackl-Kropf-Prodinger:2017:iterat-cuttin}.}

\begin{document}

\begin{abstract}
Rooted plane trees are reduced by four different operations on the
fringe. The number of surviving nodes after reducing the tree
repeatedly for a fixed
number of times is asymptotically analyzed. The four different operations include
cutting all or only the leftmost leaves or maximal paths. This generalizes the concept
of pruning a tree.

The results include exact expressions and asymptotic expansions 
for the expected value and the variance as well as central limit theorems.
\end{abstract}
\maketitle

\section{Introduction}
Plane trees are among the most interesting elementary combinatorial objects; they
appear in the literature under many different names such as ordered trees, planar trees,
planted plane trees, etc. They have been analyzed under various aspects, especially
due to their relevance in Computer Science. Two particularly well-known quantities are
the height, since it is
equivalent to the stack size needed to explore binary (expression) trees, and the pruning
number (pruning index), since it is equivalent to the register function (Horton-Strahler
number) of binary trees. Several results for the height of plane trees can be found
in \cite{Bruijn-Knuth-Rice:1972,Flajolet-Odlyzko:1982, Prodinger:1983:height-planted}, for the register
function, we refer to \cite{Vauchaussande:1985:thesis,
  Flajolet-Raoult-Vuillemin:1979:register, Kemp:1980:note-stack-size}, and for results on
the connection between the register function and the  pruning number to
\cite{Vauchaussande:1985:thesis, Zeilberger:1990:bijection}.

Reducing (cutting-down) trees has also been a popular research theme during the last decades 
\cite{Janson:2006:random-cutting, Meir-Moon:1970:cutting-random-trees,
  Panholzer:2006:cutting-trees}: according to a certain probabilistic model, a given
tree is reduced until a certain condition is satisfied (usually, the root is isolated).

In the present paper, the point of view is slightly different, as we
reduce a tree in a
completely deterministic fashion at the leaves until the tree has no
more edges. All these reductions take place on the fringe, meaning that only (a subset of)
leaves (and some adjacent structures) are removed. We consider four different models:

\begin{itemize}
\item[--] In one round, all leaves together with the corresponding edges are
  removed (see Section~\ref{sec:cutting-leaves}).
\item[--] In one round, all maximal paths (linear graphs), with the leaves on one end, are
  removed (see Section~\ref{sec:cutting-paths}). This process is called pruning.
\item[--] A leaf is called an old leaf if it is the leftmost sibling of its parents. This
  concept was introduced in \cite{Chen-Deutsch-Elizalde:2006:old-leaves}. In one round,
  only old leaves are removed (see Section~\ref{sec:cutting-old-leaves}).
\item[--] The last model deals with pruning old paths. There might be several
  interesting models related to this; the one we have chosen here is that in one round
  maximal paths are removed, under the condition that each of their nodes is the leftmost
  child of their parent node (see Section~\ref{sec:cutting-old-paths}).
\end{itemize}

The four tree reductions are illustrated in
Figure~\ref{fig:reduction-illustration}. We describe these reductions
more formally in the corresponding sections.

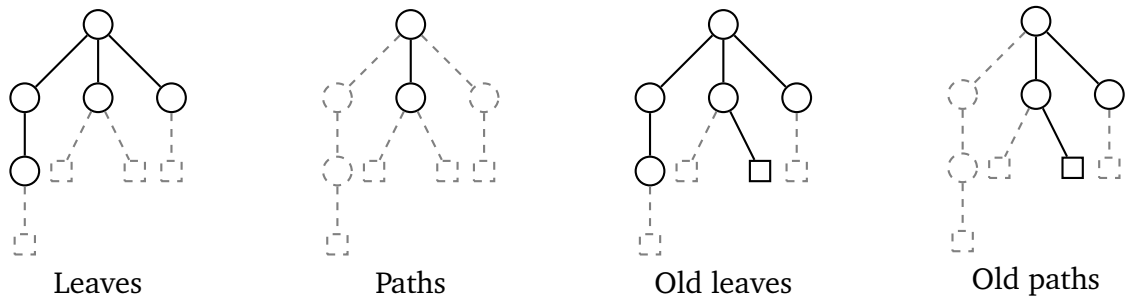
\begin{figure}[ht]
  \centering
  \begin{center}
    \begin{minipage}{0.23\linewidth}
      \centering
      \begin{tikzpicture}[thick, scale=0.65]
        \node[draw, circle] {}
        child {node[draw, circle] {} child {node[draw, circle] {}
            child[gray, dashed] {node[draw,
              rectangle] {}}}}
        child {node[draw, circle] {}
          child[gray, dashed] {node[draw, rectangle] {}}
          child[gray, dashed] {node[draw, rectangle] {}}
        }
        child {node[draw, circle] {} child[gray, dashed] {node[draw, rectangle] {}}};
      \end{tikzpicture}\\
      Leaves
    \end{minipage}\hfill
    \begin{minipage}{0.23\linewidth}
      \centering
      \begin{tikzpicture}[thick, scale=0.65]
        \node[draw, circle] {}
        child[gray, dashed] {node[draw, circle] {} child {node[draw, circle] {} child{node[draw,
              rectangle] {}}}}
        child {node[draw, circle] {}
          child[gray, dashed] {node[draw, rectangle] {}}
          child[gray, dashed] {node[draw, rectangle] {}}
        }
        child[gray, dashed] {node[draw, circle] {} child {node[draw, rectangle] {}}};
      \end{tikzpicture}\\
      Paths
    \end{minipage}\hfill
    \begin{minipage}{0.23\linewidth}
      \centering
      \begin{tikzpicture}[thick, scale=0.65]
        \node[draw, circle] {}
        child {node[draw, circle] {} child {node[draw, circle] {}
            child[gray, dashed] {node[draw,
              rectangle] {}}}}
        child {node[draw, circle] {}
          child[gray, dashed] {node[draw, rectangle] {}}
          child[] {node[draw, rectangle] {}}
        }
        child {node[draw, circle] {} child[gray, dashed] {node[draw, rectangle] {}}};
      \end{tikzpicture}\\
      Old leaves
    \end{minipage}\hfill
    \begin{minipage}{0.23\linewidth}
      \centering
      \begin{tikzpicture}[thick, scale=0.65]
        \node[draw, circle] {}
        child[gray, dashed] {node[draw, circle] {} child {node[draw, circle] {} child{node[draw,
              rectangle] {}}}}
        child {node[draw, circle] {}
          child[gray, dashed] {node[draw, rectangle] {}}
          child {node[draw, rectangle] {}}
        }
        child {node[draw, circle] {} child[gray, dashed] {node[draw, rectangle] {}}};
      \end{tikzpicture}\\
      Old paths
    \end{minipage}
  \end{center}
  \caption{Removal of (old) leaves / paths}
  \label{fig:reduction-illustration}
\end{figure}

The first model is clearly related to the height of the plane tree, and the second one to
the Horton-Strahler number via the pruning index \cite{Zeilberger:1990:bijection,
  Viennot:2002:strah-dyck}. While there are no surprises about the number of rounds
that the process takes here, we are interested in how the fringe develops. The number of leaves
and nodes altogether in the remaining tree after a  fixed number of reduction rounds is
the main parameter analyzed in this paper.

For the sake of simplicity, we will use the same notation for each of the following
reduction analyses. In case we need to compare objects from two different sections, we
will distinguish them by adding appropriate superscripts.

The random variable $X_{n,r}$  models the tree size after reducing a plane tree
of size $n$ (that is chosen uniformly at random among all trees with $n$ nodes) $r$-times
iteratively according to one of our four reductions. If a tree does not ``survive'' $r$
rounds of reductions, we consider the size of the resulting tree to be $0$. In particular,
for  $r=0$, the given plane tree is not changed and $X_{n,0}=n$.

As we will see, a key aspect of the analysis of $X_{n,r}$ is the translation of the
algorithmic description of the reduction into an operator $\Phi$ that acts on the corresponding
generating functions.

In Section~\ref{sec:cutting-leaves}, the reduction cutting away all leaves from the tree is
discussed. Section~\ref{sec:cut-leaves:preliminaries} contains all necessary auxiliary
concepts required in order to study the $r$-fold application of this
reduction. In
Section~\ref{sec:cut-leaves:expansion}, we determine the operator $\Phi$ acting on the
corresponding generating function explicitly and prove some direct consequences. Then, in
Section~\ref{sec:cut-leaves:analysis} we carry out the analysis of the behavior of
$X_{n,r}$ by computing explicit expressions and asymptotic expansions for the factorial
moments of $X_{n,r}$ as well as a central limit theorem.

Section~\ref{sec:cutting-paths} is devoted to the study of the
reduction that cuts away all
paths. As we will see in Section~\ref{sec:paths:results}, we can actually obtain all
results regarding the behavior of $X_{n,r}$ as consequences of the corresponding results
in Section~\ref{sec:cutting-leaves}. In Section~\ref{sec:paths:totalnumber}, we analyze the
asymptotic behavior of the expected number of paths required to construct a plane tree of
size $n$, i.e.\ the number of paths we can cut away until the tree cannot be reduced any
further.

Sections~\ref{sec:cutting-old-leaves} and~\ref{sec:cutting-old-paths} are devoted to the
analysis of reductions removing only leftmost leaves and leftmost paths from the tree,
respectively. In particular, in Section~\ref{sec:old-paths:total}, we study the total
number of old paths that can be removed from a tree until it cannot be reduced any
further.

On a general note, the computationally heavy parts of this paper have been carried out
with the open-source computer mathematics system SageMath~\cite{SageMath:2016:7.4}, and
the corresponding worksheets are available for download. In particular, there are the
following files:
\begin{itemize}
\item \href{https://benjamin-hackl.at/downloads/treereductions/treereductions.ipynb}{\texttt{treereductions.ipynb}}
  for most of the asymptotic computations in Sections~\ref{sec:cutting-leaves},
  \ref{sec:cutting-paths}, and \ref{sec:cutting-old-leaves},
\item \href{https://benjamin-hackl.at/downloads/treereductions/old_paths.ipynb}{\texttt{old\_paths.ipynb}}
  for most of the asymptotic computations in Section~\ref{sec:cutting-old-paths},
\item \href{https://benjamin-hackl.at/downloads/treereductions/factorial_moments_leaves.ipynb}{\texttt{factorial\_moments\_leaves.ipynb}}
  for computation of the factorial moments in Theorem~\ref{thm:cut-leaves},
\item \href{https://benjamin-hackl.at/downloads/treereductions/factorial_moments_old_paths.ipynb}{\texttt{factorial\_moments\_old\_paths.ipynb}}
  for computation of the factorial moments in Theorem~\ref{thm:moments-asy-old-paths}.
\end{itemize}
Additionally, in order to run these computations yourself, you also need to download the
following two utility files:
\begin{itemize}
\item \href{https://benjamin-hackl.at/downloads/treereductions/identities_common.py}{\texttt{identities\_common.py}},
\item \href{https://benjamin-hackl.at/downloads/treereductions/conditional_substitution.py}{\texttt{conditional\_substitution.py}}.
\end{itemize}
All these files including some instructions on how to use them can be found at \url{https://benjamin-hackl.at/publications/treereductions/}.

\section{Cutting Leaves}\label{sec:cutting-leaves}

\subsection{Preliminaries}\label{sec:cut-leaves:preliminaries}

In this part of the paper we investigate the effect of the tree
reduction that cuts away all
leaves from a given tree. However, before we can do so, we require some auxiliary concepts,
which we discuss in this section. Most importantly, we need a generating function counting
plane trees with respect to their number of inner nodes and leaves, which is intimately
linked to Narayana numbers. The generating function presented in the following proposition
is actually well-known (see, e.g.~\cite[Example III.13]{Flajolet-Sedgewick:ta:analy}).

\begin{proposition}\label{proposition:gf-narayana}
  The generating function $T(z,t)$ which enumerates plane trees with respect to their
  internal nodes (marked by the variable $z$) and leaves (marked by $t$) is given
  explicitly by
  \begin{equation}\label{eq:rootedplane:ogf}
    T(z,t) = \frac{1 - (z-t) - \sqrt{1 - 2(z+t) + (z-t)^{2}}\,}{2}.
  \end{equation}
\end{proposition}
\begin{proof}
  This can be obtained directly from the symbolic equation describing the
  combinatorial class of plane trees $\mathcal{T}$, which is illustrated in
  Figure~\ref{fig:rootedplane:symbolic}. In particular, $\square$ and $\innernode$
  represent leaves and internal nodes, respectively.
  \begin{figure}[ht]
    \centering
    \begin{tikzpicture}
      \node (add) {$\mathcal{T}\quad=\quad\square\quad+\quad\sum\limits_{n\geq1}$};
      \node[right of=add, circle, fill=black, inner sep=3pt, xshift=7em, yshift=2em] (right-V) {};
      \node[below of=right-V, yshift=-1.5em, xshift=-4em] (1) {$\mathcal{T}$};
      \node[below of=right-V, yshift=-1.5em, xshift=-2em] (2) {$\mathcal{T}$};
      \node[below of=right-V, yshift=-1.5em, xshift=0em]  (3) {$\mathcal{T}$};
      \node[below of=right-V, yshift=-1.5em, xshift=2em, gray] (4) {$\cdots$};
      \node[below of=right-V, yshift=-1.5em, xshift=4em] (5) {$\mathcal{T}$};
      \draw (right-V) -- (1) (right-V) -- (2) (right-V) -- (3) (right-V) -- (5);
      \draw[dotted, gray] (right-V) -- (4);
      \draw [
        thick,
        decoration={
        brace,
        mirror,
        raise=1em
        },
      decorate
      ] (1.center) to node (h) {} (5.center);
      \node [below of=h] {$n$};
    \end{tikzpicture}
    \caption{Symbolic equation for plane trees}
    \label{fig:rootedplane:symbolic}
  \end{figure}

  The symbolic equation translates into the functional equation
  \[ T(z,t) = t + \frac{z T(z,t)}{1 - T(z,t)},  \]
  which yields~\eqref{eq:rootedplane:ogf} after solving for $T(z,t)$ and choosing the
  appropriate branch.
\end{proof}

In the context of plane trees, the so-called \emph{Narayana numbers} count the
number of trees with a given size and a given number of leaves
(cf.~\cite{Drmota:2015:trees-handbook}). As these numbers will appear throughout the
entire paper, we introduce them formally and investigate some properties within the
following statements.

\begin{definition}\label{def:narayana-poly}
  The \emph{Narayana numbers} are defined as
  \begin{equation*}
    N_{n,k}=\frac{1}{n}\binom{n}{k-1}\binom{n}{k}
  \end{equation*}
for $1\leq n$ and  $1\leq k\leq n$, and $N_{0,0}=1$. All other indices give
$N_{n,k}=0$. Combinatorially, for $n\geq 1$ the Narayana number $N_{n,k}$ corresponds to the number of
plane trees with $n$ edges (i.e.\ $n+1$ nodes) and $k$ leaves.
  The \emph{Narayana polynomials} are defined as
  \begin{equation*}
    N_{n}(x)=\sum_{k=1}^{n}N_{n,k}x^{k-1}
  \end{equation*}
for $n\geq1$ and $N_{0}(x)=1$, and the \emph{associated Narayana polynomials} are defined
as
\[ \tilde{N}_{n}(x) = x\cdot N_{n}(x)\]
for $n\geq 0$. Note that
\begin{equation*}
N_{n}(1) = \tilde{N}_{n}(1) = C_{n} = \frac{1}{n+1}\binom{2n}{n}
\end{equation*}
is the $n$th Catalan number.
\end{definition}

\begin{remark}
  The generating function $\frac{1}{z} T(z,z) = \frac{1 - \sqrt{1 - 4z}\,}{2z}$ enumerates Catalan
  numbers, see~\cite[Theorem~3.2]{Drmota:2009:random}, and
  the generating function $T(z,tz)$ enumerates Narayana numbers
  \begin{equation}\label{eq:gf-narayana}
    T(z,tz)=zt+\sum_{n\geq2}\sum_{k=1}^{n-1}N_{n-1,k}z^{n}t^{k} = \sum_{n\geq 1} z^{n} \tilde{N}_{n-1}(t).
  \end{equation}
  We will frequently use this relation in the form
  \begin{equation}\label{eq:gf-narayana:quotient}
    T(z, t) = \sum_{n\geq 1} z^{n} \tilde{N}_{n-1}\Bigl(\frac tz\Bigr).
  \end{equation}
  Furthermore, it is easily checked that $T(z, tz)$ satisfies the ordinary differential
  equation
  \[ (1  - 2(t+1)z + (1-t)^{2} z^{2}) \frac{\partial}{\partial z} T(z, tz) - ((1-t)^{2} z -
    t - 1) T(z, tz) = t (1 + z - tz). \]
  Extracting the coefficient of $z^{n+2}$ then yields the recurrence relation
  \begin{equation}
    \label{eq:narayana-recursion}
    (n+3) \tilde{N}_{n+2}(t) - (2n+3)(t+1)\tilde{N}_{n+1}(t) + n (t-1)^{2} \tilde{N}_{n}(t) = 0
  \end{equation}
  for $n\ge 0$.
\end{remark}

The following proposition gives another useful property of associated Narayana polynomials.
\begin{proposition}\label{prop:narayana-reverse}
  Let $n\ge 0$, then we have the relation
  \begin{equation}\label{eq:Narayana-reverse}
    t^{n+1}\tilde N_n\Bigl(\frac1t\Bigr) = (1-t)\iverson{n=0} + \tilde N_n(t).
  \end{equation}
\end{proposition}
\begin{proof}
  This relation follows from extracting the coefficient of $z^{n+1}$ from the identity $T(tz, z) = T(z,
  tz) + (1-t)z$ with the help of~\eqref{eq:gf-narayana:quotient}.

  While it is straightforward to prove that the identity is valid by means of
  algebraic manipulation, we also give a combinatorial proof.

  From a combinatorial point of view, both generating functions $T(tz, z)$ and $T(z, tz)$
  enumerate plane trees where $z$ marks the tree size, the only difference is that
  the variable $t$ enumerates inner nodes in $T(tz, z)$ and leaves in $T(z,tz)$. We want
  to show that for trees of size $n \geq 2$, these two classes are equal, resulting in
  $T(tz, z) - z = T(z, tz) - tz$.

  To construct an appropriate bijection between the class of trees of size $n$ with $k$
  leaves and the class of trees of size $n$ with $k$ inner nodes we need to have a closer
  look at the well-known \emph{rotation
    correspondence}~\cite[I.5.3]{Flajolet-Sedgewick:ta:analy}, which is a bijection
  between plane trees of size $n$ and binary trees with $n-1$ inner nodes. In fact, the
  leaves in the binary tree are strongly related to the leaves and inner nodes of the
  original tree:
  \begin{itemize}
  \item[--] Left leaves in the binary tree are only attached to those nodes whose
    companions in the plane tree have no children, i.e., to those who correspond to
    leaves in the plane tree.
  \item[--] Right leaves, on the other hand, are attached to nodes whose companion nodes
    in the plane tree have no sibling right of them. This means that for every node with
    children, i.e., for every inner node, there is precisely one rightmost child and thus
    precisely one right leaf in the binary tree.
  \end{itemize}
  The bijection between the two tree classes can now be described as follows: given some
  tree of size $n$ and $k$ leaves, apply the rotation correspondence in order to obtain a
  binary tree. Then mirror the binary tree by swapping all left and right
  children. Transform this mirrored tree back by means of the inverse rotation
  correspondence, and the result is a plane tree of size $n$ and $k$ inner nodes as
  mirroring the binary tree swapped the number of left and right leaves in the tree. This
  proves the proposition.
\end{proof}

Derivatives of the associated Narayana polynomials defined above will occur within the
analysis of a reduction model later, which is why we compute some special values in the
following proposition.
\begin{proposition}\label{prop:narayana-derivative}
  Evaluating the $r$th derivative of the associated Narayana polynomials at $1$, i.e.\
  $\tilde{N}_{n}^{(r)}(1)$, gives the number of trees with $n+1$ nodes where precisely $r$
  leaves are selected and labeled from $1$ to $r$. In particular, for $n\geq 1$ we have
  \[ \tilde{N}_{n}'(1) = \frac{1}{2}\binom{2n}{n},\qquad \tilde{N}_{n}''(1) =
    (n-1)\binom{2n-2}{n-1}. \]
\end{proposition}
\begin{proof}
  The combinatorial interpretation follows immediately by rewriting
  \[ \tilde{N}_{n}^{(r)}(1) = \sum_{k=1}^{n} N_{n,k} k^{\underline{r}},  \]
  where we used the notion $k^{\underline{r}} = k (k-1) \cdots (k-r+1)$ for the falling
  factorial.
  Explicit values can be obtained by differentiating~\eqref{eq:gf-narayana} $r$-times with
  respect to $t$, then setting $t=1$ and extracting the coefficient of $z^{n+1}$.
\end{proof}
\begin{remark}
  By the combinatorial interpretation of Proposition~\ref{prop:narayana-derivative} we
  find that $\tilde{N}_{n}'(1) = \frac{1}{2}\binom{2n}{n}$ enumerates
  the number of leaves, summed over all trees with $n+1$ nodes. At the same time, as there
  are $C_{n} = \frac{1}{n+1} \binom{2n}{n}$ such trees, the total number of nodes in these
  trees is $\binom{2n}{n}$. This implies that exactly half of all nodes in all trees of given size
  are leaves!

  In fact, this interpretation also motivates a second, purely combinatorial proof of the
  explicit value of $\tilde{N}_{n}'(1)$: the bijection correspondence maps trees of size
  $n+1$ to binary trees with $n$ inner nodes. In the proof of
  Proposition~\ref{prop:narayana-reverse} we already observed that the number of left
  leaves in the binary tree obtained from the rotation correspondence is equal to the
  number of leaves in the plane tree.

  As binary trees with $n$ inner nodes have $n+1$ leaves, and as there are $C_{n}$ binary
  trees with $n$ inner nodes, the total number of leaves in all binary trees with $n$
  inner nodes is $\binom{2n}{n}$. By symmetry, there have to be equally many left leaves
  as right leaves---which proves that there are $\frac{1}{2}\binom{2n}{n}$ left leaves,
  and thus $\tilde{N}_{n}'(1) = \frac{1}{2}\binom{2n}{n}$.
\end{remark}

In addition to the polynomials related to the Narayana numbers, there is another well-known sequence
of polynomials that will occur throughout this paper.

\begin{definition}\label{def:fib-poly}
  The \emph{Fibonacci polynomials} are recursively defined by
\begin{equation*}
  F_{r}(z)=F_{r-1}(z)+zF_{r-2}(z)
\end{equation*}
for $r\geq 2$ and $F_{0}(z)=0$, $F_{1}(z)=1$.
\end{definition}

For many identities involving Fibonacci numbers, there is an analogous statement for
Fibonacci polynomials. The identity presented in the following proposition will be used
repeatedly throughout this paper.
\begin{proposition}[d'Ocagne's Identity]
  Let $s, r \in \mathbb{Z}_{\geq 0}$ where $s \geq r$. Then we have
  \begin{equation}
    \label{eq:fibonacci:general}
    F_{r+1}(z) F_{s}(z) - F_{r}(z) F_{s+1}(z) = (-z)^{r} F_{s-r}(z).
  \end{equation}
\end{proposition}
\begin{proof}
  The left-hand side of~\eqref{eq:fibonacci:general} can be expressed as the determinant
  of $\begin{pmatrix} F_{r+1}(z) & F_{r}(z) \\ F_{s+1}(z) & F_{s}(z) \end{pmatrix}$. At
  the same time, for $r$, $s \geq 1$  we can write
  \[ \begin{pmatrix} F_{r+1}(z) & F_{r}(z) \\ F_{s+1}(z) & F_{s}(z) \end{pmatrix} =
    \begin{pmatrix} F_{r}(z) & F_{r-1}(z) \\ F_{s}(z) &
      F_{s-1}(z) \end{pmatrix} \begin{pmatrix} 1 & 1 \\ z & 0 \end{pmatrix}.\]
  Combining these two observations yields
  \[ F_{r+1}(z) F_{s}(z)- F_{r}(z) F_{s+1}(z) = \det\begin{pmatrix} F_{r+1}(z) & F_{r}(z)
      \\ F_{s+1}(z) & F_{s}(z) \end{pmatrix} = \det\begin{pmatrix} 1 & 0 \\ F_{s+1-r}(z)
      & F_{s-r}(z) \end{pmatrix} \det\begin{pmatrix} 1 & 1 \\ z & 0 \end{pmatrix}^{r},  \]
  which proves the statement.
\end{proof}

Observe that setting $s = r+1$ in~\eqref{eq:fibonacci:general} yields the identity
\begin{equation}
  \label{eq:fibonacci}
  F_{r+1}(z)^{2} - F_{r}(z) F_{r+2}(z) = (-z)^{r},
\end{equation}
which we will make heavy use of later on.

An important tool in the context of plane trees is the substitution $z = u/(1+u)^{2}$,
which allows us to write some expressions in a manageable form. It
is easy to check that with this
substitution, we can write Fibonacci polynomials as
\begin{equation}\label{eq:fib-poly-u}
  F_{r}(-z)=\frac{1-u^{r}}{(1-u)(1+u)^{r-1}}.
\end{equation}

The fact that this substitution also works for Fibonacci polynomials is not that
surprising, as $zF_{r}(-z)/F_{r+1}(-z)$ is the generating function of
plane trees with height $\leq r$ (see~\cite{Bruijn-Knuth-Rice:1972}).

\subsection{Leaf-Reduction and the Expansion Operator}\label{sec:cut-leaves:expansion}

The reduction $\rho\colon \mathcal{T}\setminus\{\innernode\} \to \mathcal{T}$ we want to
investigate now can be explained very easily. For any tree $\tau\in
\mathcal{T}\setminus\{\innernode\}$ we obtain the reduced tree $\rho(\tau)$ simply by
removing all leaves from $\tau$. Repeated application of $\rho$ to a tree is illustrated
in Figure~\ref{fig:cut-leaves:illustration}.

\begin{figure}[ht]
  \centering
      \[
    \begin{tikzpicture}[thick, scale=0.75, baseline={([yshift=-0.5em]current bounding
      box.center)}]
      \node[draw, circle] {}
      child {node[draw, circle] {} child {node[draw, circle] {}
          child[gray, dashed] {node[draw,
            rectangle] {}}}}
      child {node[draw, circle] {}
        child[gray, dashed] {node[draw, rectangle] {}}
        child[gray, dashed] {node[draw, rectangle] {}}
      }
      child {node[draw, circle] {} child[gray, dashed] {node[draw, rectangle] {}}};
    \end{tikzpicture}
    \quad \mapsto \quad
    \begin{tikzpicture}[thick, scale=0.75, baseline={([yshift=-0.5em]current bounding
      box.center)}]
      \node[draw, circle] {}
      child {node[draw, circle] {} child[gray, dashed] {node[draw, rectangle] {}}}
      child[gray, dashed] {node[draw, rectangle] {}}
      child[gray, dashed] {node[draw, rectangle] {}};
    \end{tikzpicture}
    \quad \mapsto \quad
    \begin{tikzpicture}[thick, scale=0.75, baseline={([yshift=-0.5em]current bounding
      box.center)}]
      \node[draw, circle] {}
      child[gray, dashed] {node[draw, rectangle] {}};
    \end{tikzpicture}\quad \mapsto \quad
    \begin{tikzpicture}[thick, scale=0.75, baseline={([yshift=-0.5em]current bounding
      box.center)}]
      \node[draw, rectangle] {};
    \end{tikzpicture}
    \]
  \caption{Illustration of the ``cutting leaves''-operator $\rho$}
  \label{fig:cut-leaves:illustration}
\end{figure}
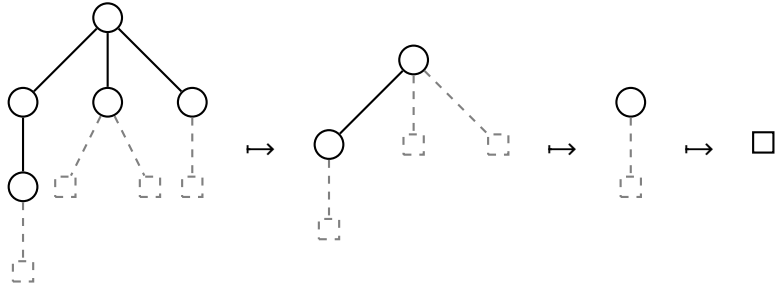

It is easy to see that this operator is certainly not injective: there are many trees that
reduce to the same tree. However, it is also easy to see that $\rho$ is surjective, as we
can always construct an expanded tree that reduces to any given tree $\tau$ by attaching
leaves to all leaves of $\tau$.

In fact, the operator $\rho^{-1}$ mapping trees $\tau\in \mathcal{T}$ to the set of
preimages is easier to handle from a combinatorial point of view. This is because we can model
the expansion of trees in the language of generating functions.

\begin{proposition}\label{prop:cut-leaves:expansion}
  Let $\mathcal{F} \subseteq \mathcal{T}$ be a family of plane trees with bivariate
  generating function $f(z,t)$, where $z$ marks inner nodes and $t$ marks leaves. Then the
  generating function of $\rho^{-1}(\mathcal{F})$,
  the family of trees whose reduction is in $\mathcal{F}$, is given by
  \begin{equation}\label{eq:cut-leaves-op}
    \Phi(f(z,t)) = (1-t) f\Big(\frac{z}{(1-t)^{2}}, \frac{zt}{(1-t)^{2}}\Big).
  \end{equation}
\end{proposition}
\begin{proof}
  It is obvious from a combinatorial point of view that the operator $\Phi$ has to be
  linear. Thus we only have to determine how a tree represented by an arbitrary monomial
  $z^{n}t^{k}$, i.e.\ a tree $\tau$ with $n$ inner nodes and $k$ leaves, is expanded.

  In order to obtain all possible tree expansions from $\tau$, we perform the following
  operations: first, all leaves of $\tau$ are expanded by appending a nonempty sequence of leaves to
  each of them. Then, every inner node of $\tau$ is expanded by appending (possibly empty)
  sequences of leaves between two of its children as well as before the first and after
  the last one.

  In terms of generating functions, expanding the leaves of $\tau$ corresponds to
  replacing $t$ by $zt/(1-t)$. Expanding the inner vertices is a bit more involved: by considering
  that every inner node has precisely one more available position to attach new leaves
  than it has children we find that there are $2n+k-1$ available positions overall within
  $\tau$. Therefore we find
  \[ \Phi(z^{n}t^{k}) = z^{n} \Big(\frac{zt}{1-t}\Big)^{k} \frac{1}{(1-t)^{2n+k-1}},  \]
  which, as $\Phi$ is linear, immediately proves~\eqref{eq:cut-leaves-op}.
\end{proof}

\begin{corollary}
  The generating function for plane trees $T(z,t)$ satisfies the functional
  equation
  \begin{equation}\label{eq:cut-leaves:functional}
    T(z,t) = t + \Phi(T(z,t)).
  \end{equation}
\end{corollary}
\begin{proof}
  This follows directly from the fact that $\rho\colon \mathcal{T}\setminus \{\innernode\}
  \to \mathcal{T}$ is surjective, i.e. $\rho^{-1}(\mathcal{T}) =
  \mathcal{T}\setminus\{\innernode\}$.
\end{proof}
\begin{corollary}
  The Narayana numbers satisfy the identity
  \begin{equation*}
    N_{n+k-1,k}=\sum_{\ell=1}^{k}\binom{2n+k-\ell-2}{k-\ell}N_{n-1,\ell}
  \end{equation*}
  for $n\geq 2$, $k\geq1$.
\end{corollary}
\begin{proof}
  The result follows from extracting the coefficient of $z^{n}t^{k}$ from both sides of~\eqref{eq:cut-leaves:functional}.
\end{proof}
\begin{remark}
  Note that in~\cite{Callan:2012:kreweras-narayana-identity} there is a very short proof
  based on Dyck paths for this identity, and actually the
  argumentation there is strongly related to our tree reduction here: by the well-known
  glove bijection, it is easy to see that cutting away all leaves of a plane tree
  translates into removing all peaks within the corresponding Dyck path.
\end{remark}

We are now interested in determining a multivariate generating function enumerating
plane trees with respect to the tree size as well as the size of the tree after applying
the tree reduction $\rho$ a fixed number of times.

\begin{proposition}\label{prop:cut-leaves:gf}
  Let $r\in \N_{0}$. The trivariate generating function $G_{r}(z,v_{I},v_{L}) = G_{r}^{\operatorname{L}}(z,v_{I},v_{L})$ enumerating
  plane trees whose leaves can be cut at least $r$-times, where $z$ marks the tree size,
  and $v_{I}$ and $v_{L}$ mark the number of inner nodes and leaves of the $r$-fold cut
  tree, respectively, is given by
  \begin{equation}\label{eq:cut-leaves:iterated-gf}
    G_{r}(z,v_{I},v_{L}) = \Phi^{r}(T(zv_{I},tv_{L}))|_{t=z} = \frac{1 - u^{r+2}}{(1-u^{r+1})(1+u)} T\Big(\frac{u
      (1-u^{r+1})^{2}}{(1-u^{r+2})^{2}} v_{I}, \frac{u^{r+1}(1-u)^{2}}{(1-u^{r+2})^{2}} v_{L}\Big).
  \end{equation}
\end{proposition}
\begin{proof}
  First, observe that formally, we can obtain the generating function enumerating
  plane trees that can be reduced at least $r$-times with respect to their size by
  considering $\Phi^{r}(T(z,t))|_{t=z}$. If we additionally track some size parameter like
  the number of inner nodes or the number of leaves
  before the expansion by marking their size with $v_{I}$ and $v_{L}$, then we obtain a generating function
  for plane trees that can be reduced at least $r$-times where $v_{I}$ and $v_{L}$ mark
  inner nodes and leaves in the original tree and $z$ marks the size of the expanded
  tree. From a different point of view,
  $z$ marks the size of the original tree and $v_{I}$ and $v_{L}$ mark the number of
  inner nodes and leaves of the $r$-fold
  reduced tree, meaning that we have
  \[ G_{r}(z,v_{I},v_{L}) = \Phi^{r}(T(zv_{I},tv_{L}))|_{t=z}, \]
  which proves the first equation in~\eqref{eq:cut-leaves:iterated-gf}.

  As $\Phi$ is linear, we are mainly interested in finding a representation for
  $\Phi^{r}(z^{n}t^{k})|_{t=z}$. To do so, we consider the strongly related operator
  \[ \Psi(f(z,t)) \coloneqq f\Big(\frac{z}{(1-t)^{2}}, \frac{zt}{(1-t)^{2}}\Big). \]
  It is easy to prove by induction that iterative application of $\Phi$ can be expressed
  in terms of $\Psi$ via
  \[ \Phi^{r}(f(z,t)) = \Psi^{r}(f(z,t)) \prod_{j=0}^{r-1} (1 - \Psi^{j}(t)),  \]
  which means that we can concentrate on the investigation of the linear operator
  $\Psi$. Note that $\Psi$ is also multiplicative, meaning that $\Psi^{r}(z^{n}t^{k}) =
  \Psi^{r}(z)^{n} \Psi^{r}(t)^{k}$.

  Again by induction, it is easy to show that the recurrences
  \[ \Psi^{r+1}(t) = \frac{z \Psi^{r}(t)}{\prod_{j=0}^{r} (1 - \Psi^{j}(t))^{2}}
    \quad\text{and}\quad
    \Psi^{r+1}(z) = \frac{z}{\prod_{j=0}^{r} (1 - \Psi^{j}(t))^{2}} \]
  hold for $r\geq 0$. Now define $f_{r} \coloneqq \Psi^{r}(t)|_{t=z}$ and $g_{r}\coloneqq
  \Psi^{r}(z)|_{t=z}$. We prove by induction that these quantities can be
  represented by means of Fibonacci polynomials as
  \[ f_{r} = \frac{z^{r+1}}{F_{r+2}(-z)^{2}} \quad\text{and}\quad g_{r} = \frac{z
      F_{r+1}(-z)^{2}}{F_{r+2}(-z)^{2}}  \]
  for $r\geq 0$, where the recurrence relations from above, the
  identity~\eqref{eq:fibonacci} as well as the relation
  \[ \prod_{j=0}^{r-1} (1-f_{j}) = \frac{F_{r+2}(-z)}{F_{r+1}(-z)}  \]
  for $r\geq 0$ play integral parts in the proof.

  With these explicit representations, we find
  \begin{equation}\label{eq:cut-leaves:iteratedPhi-fibo}
    \Phi^{r}(z^{n}t^{k})|_{t=z} = \Psi^{r}(z^{n}t^{k})|_{t=z} \prod_{j=0}^{r-1}(1-f_{j})
    = \frac{z^{n+k(r+1)} F_{r+1}(-z)^{2n-1}}{F_{r+2}(-z)^{2n+2k-1}}.
  \end{equation}
  Then, using~\eqref{eq:fib-poly-u} and rewriting the right-hand side
  of~\eqref{eq:cut-leaves:iteratedPhi-fibo} in terms of $u$, where $z = u/(1+u)^{2}$,
  yields
  \[ \Phi^{r}(z^{n}t^{k})|_{t=z} = \frac{1 - u^{r+2}}{(1-u^{r+1})(1+u)} \Big(\frac{u
      (1-u^{r+1})^{2}}{(1 - u^{r+2})^{2}}\Big)^{n} \Big(\frac{u^{r+1}
      (1-u)^{2}}{(1-u^{r+2})^{2}}\Big)^{k}.  \]
  By linearity, we are allowed to apply $\Phi^{r}$ to every summand in the power
  series expansion of $f(z,t)$ separately---which proves the statement.
\end{proof}

The generating function $G_{r}(z,v,v)$ tells us how many nodes (marked by $v$) are
still in the tree after $r$ reductions. For the sake of brevity we set $G_{r}(z,v) \coloneqq
G_{r}(z,v,v)$. It is completely described in
terms of the function $T(z,t)$, although in a non-trivial way. Results
about moments and the limiting distribution can be extracted from this explicit form.

With the help of the mathematics software system SageMath~\cite{SageMath:2016:7.4}, the
generating function $G_{r}(z,v)$ can be expanded. For small values of $r$, the first few
summands are
\[ G_{1}(z,v) = v z^{2} + (v^{2} + v) z^{3} + (v^{3} + 3v^{2} + v) z^{4} + (v^{4} + 6v^{3}
  + 6v^{2} + v) z^{5} + O(vz^{6}), \]
\[ G_{2}(z,v) = vz^3 + (v^2 + 3v)z^4 + (v^3 + 5v^2 + 7v)z^5 + (v^4 + 7v^3 + 18v^2 +
  15v)z^6 + O(vz^{7}),  \]
\[ G_{3}(z,v) = vz^4 + (v^2 + 5v)z^5 + (v^3 + 7v^2 + 18v)z^6 + (v^4 + 9v^3 + 33v^2 +
  57v)z^7 + O(vz^{8}).  \]

As announced in the introduction, we investigate the behavior of the random
variable $X_{n,r} = X_{n,r}^{\operatorname{L}}$ that models the number of nodes which are left
after reducing a random tree $\tau$ with $n$ nodes $r$-times. In case the $r$-fold
application of $\rho$ to $\tau$ is not defined, we consider the resulting tree size to be
$0$, i.e., the random variable $X_{n,r} = 0$ for these trees.  Note that the tree $\tau$ is
chosen uniformly at random among all trees of size $n$. With the help of the generating
function $G_{r}(z,v)$ we are able to express the probability generating function of $X_{n,r}$ as
\begin{equation}\label{eq:pgf-cut-leaves}
  \E v^{X_{n,r}}=\frac{a_{n,r}+[z^{n}]G_{r}(z,v)}{C_{n-1}}
\end{equation}
where $a_{n,r}$ is the number of trees of size $n$ which are empty
after reducing $r$-times. We have $a_{n,r}=C_{n-1}-[z^{n}]G_{r}(z,1)$.

In addition to $X_{n,r}$, we also consider the random variables $I_{n,r}$ and $L_{n,r}$
that model the number of inner nodes and leaves, respectively, that remain after reducing
a random tree with $n$ nodes $r$ times. The generating functions corresponding to $I_{n,r}$
and $L_{n,r}$ are $G_{r}(z,v,1)$ and $G_{r}(z,1,v)$, respectively.

The relations $X_{n,r}\overset{d}{=}I_{n,r} + L_{n,r}$ and $I_{n,r} \overset{d}{=}
X_{n,r+1}$ hold by the combinatorial interpretation of the operator $\Phi$.

\subsection{Asymptotic Analysis}\label{sec:cut-leaves:analysis}
We find explicit generating functions for the factorial moments of the random variables
$X_{n,r}$, $I_{n,r}$, and $L_{n,r}$.

\begin{proposition}\label{prop:cut-leaves:moments}
  The $d$th factorial moments of $X_{n,r}$, $I_{n,r}$ and $L_{n,r}$ are given by
  \begin{equation}\label{eq:moments}
\E X_{n,r}^{\underline{d}}=\E I_{n,r-1}^{\underline{d}}=\frac{1}{C_{n-1}}[z^{n}]\frac{\partial^d }{\partial v^d}G_r(z,1)\Big|_{v=1}
=\frac{1}{C_{n-1}}[z^{n}]\frac{u^{d}d!}{(1+u)(1-u^{r+1})^d(1-u)^{d-1}}\tilde{N}_{d-1}(u^{r})
\end{equation}
and
\begin{equation}
  \label{eq:moments-fringe}
  \E L_{n,r}^{\underline{d}} = \frac{1}{C_{n-1}} [z^{n}] \frac{u^{dr+2d} (1-u)
    d!}{(1+u)(1-u^{r+2})^{d}(1-u^{r+1})^{d}} \tilde{N}_{d-1}\Bigl(\frac1u\Bigr)
\end{equation}
where $z=u/(1+u)^{2}$ for $d\in\Z_{\geq 1}$.
\end{proposition}
\begin{remark}
  For $d\ge 2$, $u^d\tilde N_{d-1}(u^{-1})$ can be replaced by $\tilde
  N_{d-1}(u)$ in \eqref{eq:moments-fringe},
  see \eqref{eq:Narayana-reverse}.
\end{remark}

\begin{proof}
  We use the abbreviations
  \begin{equation*}
    a\coloneqq\frac{u
      (1-u^{r+1})^{2}}{(1-u^{r+2})^{2}},\qquad
    b\coloneqq \frac{u^{r+1}(1-u)^{2}}{(1-u^{r+2})^{2}}, \qquad
    c\coloneqq \frac{1 - u^{r+2}}{(1-u^{r+1})(1+u)}.
  \end{equation*}
  We consider the exponential generating function of $\partial^d/(\partial v)^d
  G_r(z, v)$ to be a Taylor series and obtain
  \begin{equation*}
    \sum_{d\ge 0}\frac1{d!}\frac{\partial^d}{\partial v^d} G_r(z,
      v) q^d  = G_r(z, v+q).
  \end{equation*}

  By Proposition~\ref{prop:cut-leaves:gf}, extracting the coefficient of $q^d$ yields
  \begin{equation*}
    \frac{\partial^d}{\partial v^d} G_r(z, v)\Bigr|_{v=1} = d![q^d]G_r(z,
    1+q)=d!c[q^d]T(a(1+q), b(1+q)).
  \end{equation*}
  We have
  \begin{align*}
    T(a(1+q), b(1+q))&=
    \frac{1-(a-b)(1+q) -
    \sqrt{1-2(1+q)(a+b)+(1+q)^2(a-b)^2}\,}2\\
    &= \frac{1-(a-b)-(a-b)q}2\\
    &\qquad-
    \frac{\sqrt{1-2(a+b) +(a-b)^2 - 2q(a+b-(a-b)^2)+q^2(a-b)^2}\,}2.
  \end{align*}

  By using the fact that
  \begin{equation*}
    1-2(a+b) + (a-b)^2 = \Delta^2\quad  \text{ for }\quad \Delta = \frac{(1-u)(1-u^{r+1})}{1-u^{r+2}}
  \end{equation*}
  and by choosing $\alpha$ and $\beta$ such that
  \begin{equation*}
    \alpha+\beta = \frac{a+b-(a-b)^2}{\Delta^2},\qquad
    \alpha-\beta = \frac{a-b}{\Delta},
  \end{equation*}
  we obtain

  \begin{align*}
    T(a(1+q), b(1+q)) &= \frac{\Delta}{2}\Bigl(\frac1{\Delta}-(\alpha-\beta)-(\alpha-\beta)q -
                        \sqrt{1 - 2q(\alpha+\beta)+q^2(\alpha-\beta)^2}\,\Bigr)\\
    &=\frac{\Delta(\frac1\Delta -1-(\alpha-\beta))}2 + \Delta T(\alpha q, \beta q).
  \end{align*}
  Extracting the coefficient of $q^d$ for $d\ge 1$ yields
  \begin{equation*}
    \frac{\partial^d}{\partial v^d} G_r(z, v) \Bigr|_{v=1}
    =cd!\Delta [q^d]\sum_{d\ge 1}\alpha^dq^d \tilde N_{d-1}\Bigl(\frac{\beta}{\alpha}\Bigr)
    =cd!\Delta \alpha^d \tilde N_{d-1}\Bigl(\frac{\beta}{\alpha}\Bigr)
  \end{equation*}
  where \eqref{eq:gf-narayana:quotient} has been used.

  Noting that
  \begin{equation*}
    \alpha = \frac{u}{(1-u)(1-u^{r+1})}\quad\text{ and }\quad
    \beta = \frac{u^{r+1}}{(1-u)(1-u^{r+1})}
  \end{equation*}
  completes the proof of \eqref{eq:moments}.

  For the proof of \eqref{eq:moments-fringe}, we proceed in the same way and
  use the identity
  \begin{equation*}
    T(a, b(1+q)) = \frac{1-\Delta-(a-b)}{2} +  \Delta T(\alpha'q, \beta'q)
  \end{equation*}
  for
  \begin{equation*}
    \alpha'= \frac{u^{r+2}}{(1-u^{r+1})(1-u^{r+2})}, \qquad
    \beta'=\frac{u^{r+1}}{(1-u^{r+1})(1-u^{r+2})}.
  \end{equation*}

\end{proof}

From the proof of Proposition~\ref{prop:cut-leaves:moments}, we extract the
following identities for the modified Narayana polynomials.
\begin{remark}
  For $d\in \Z_{\geq 1}$ the power series identities
  \begin{align}
    \label{eq:series-identity1}
    \sum_{n\geq 1} \binom{n}{d} \frac{u^{n-d} (1 - ux)^{2n+d-1} (1 - u)^{d-1}}{(1 -
      u^{2}x)^{2n-1}} \tilde{N}_{n-1} \Big(\frac{x (1-u)^{2}}{(1-ux)^{2}}\Big) &= \tilde{N}_{d-1}(x)\\
    \label{eq:series-identity2}
    \sum_{n\geq 1} \frac{u^{n-2d} (1 - ux)^{2n-d-1} (1-u)^{2d-1}}{(1 - u^{2}x)^{2n-d-1} d!}
    \tilde{N}_{n-1}^{(d)}\Big(\frac{x (1-u)^{2}}{(1-ux)^{2}}\Big) &= \tilde{N}_{d-1}\Bigl(\frac1u\Bigr)
  \end{align}
  hold, where $\tilde{N}_{n-1}^{(d)}$ denotes the $d$th derivative of
  $\tilde{N}_{n-1}$.
\end{remark}
\begin{proof}
  In the proof of Proposition~\ref{prop:cut-leaves:moments}, we showed that
  \begin{equation*}
    \frac{\partial^d}{\partial v^d}cT(av, bv)\Bigr|_{v=1}=cd!\Delta \alpha^d
    \tilde N_{d-1}\Bigl(\frac{\beta}{\alpha}\Bigr).
  \end{equation*}
  Expanding the left side using \eqref{eq:gf-narayana:quotient} and evaluating
  the derivative yields~\eqref{eq:series-identity1} (where $u^r$ has been
  replaced by the independent variable $x$).

  The identity \eqref{eq:series-identity2} is proved in the same way.
\end{proof}

\begin{corollary}\label{corollary:leaves-exact-exp-val}
  The expected value of $X_{n+1,r}$ is explicitly given by
  \begin{align*}
    \E X_{n+1,r}&=\frac{1}{C_{n}}\sum_{\ell\geq1}\bigg(\binom{2n}{n+1-\ell(r+1)}-\binom{2n}{n-\ell(r+1)}\bigg).
  \end{align*}
\end{corollary}
\begin{proof}
  Using Proposition~\ref{prop:cut-leaves:moments} and Cauchy's integral formula, we have
  \begin{align*}
    C_{n}\E X_{n+1,r}&=[z^{n+1}]\frac{u^{r+1}}{(1+u)(1-u^{r+1})}\\
    &=\frac{1}{2\pi i}\oint_{\gamma}
    \frac{u^{r+1}}{(1+u)(1-u^{r+1})}\frac{dz}{z^{n+2}}\\
    &=\frac{1}{2\pi i}\oint_{\tilde\gamma}
    \frac{u^{r+1}(1-u)(1+u)^{2n}}{1-u^{r+1}}\frac{du}{u^{n+2}},
  \end{align*}
  where $\gamma$ is a circle around $0$ with a sufficiently small radius such that
  $\gamma'$, the image of $\gamma$ under the transformation, is a small contour circling
  $0$ exactly once as well.

Expanding $(1-u^{r+1})^{-1}$ into a geometric series and exchanging
integration and summation, we obtain
\begin{align*}
 C_{n} \E X_{n+1,r}&=\sum_{\ell\geq1}[u^{n+1-\ell(r+1)}](1-u)(1+u)^{2n},
\end{align*}
which implies the result.
\end{proof}

Having determined a closed form for this generating function allows us to analyze the
asymptotic behavior of $X_{n,r}$ in a relatively
straightforward way.

\begin{theorem}\label{thm:cut-leaves}
  Let $r\in\N_{0}$ be fixed and consider $n \to \infty$. Then the expected size and the corresponding variance of an
  $r$-fold cut plane tree are given by
  \begin{equation}\label{eq:cut-leaves:exp}
    \E X_{n,r} = \frac{n}{r+1} - \frac{r (r-1)}{6 (r+1)} + O(n^{-1}),
  \end{equation}
  and
  \begin{equation}\label{eq:cut-leaves:var}
    \V X_{n,r} = \frac{r (r + 2)}{6 (r+1)^{2}} n + O(1).
  \end{equation}

The factorial moments are asymptotically given by
\begin{equation*}
  \E X_{n,r}^{\underline{d}}=\frac{1}{(r+1)^{d}}n^{d}+\frac{d}{12(r+1)^{d}}(dr^{2} - 4dr -
  3r^{2} - 6d + 6r + 6)n^{d-1}+O(n^{d-3/2})
\end{equation*}
for $d\geq 1$. Note that all $O$-constants above depend implicitly on $r$.
\end{theorem}
\begin{proof}
  In a nutshell, we want to extract the growth of the derivatives of the generating functions $\frac{\partial^{d}
    }{\partial v^{d}}G_{r}(z,1)$, as dividing these quantities by $C_{n-1}$ yields the
  factorial moments. We want to extract the growth by
  means of singularity analysis (cf.~\cite{Flajolet-Odlyzko:1990:singul}).

  In order to do so, we first need to establish the location of the dominant singularity
  of these generating functions, which are explicitly given in~\eqref{eq:moments}.

  The singularities of \eqref{eq:moments} are roots of
  unity in terms of $u$. Substituting back $u=(1-\sqrt{1-4z}\,)/(2z)-1$ maps these roots
  of unity to real numbers greater or equal to $1/4$ and only $u=1$
  is mapped to $z=1/4$. Thus $z=1/4$ is the dominant singularity of
  \eqref{eq:moments}. A more detailed treatment of these analytic properties of the
  substitution $z = u/(1+u)^{2}$ can be found in~\cite[Proposition
  2.3]{Hackl-Heuberger-Prodinger:2016:reduc-binar}.

As $N_{0}(x)=1$, we obtain the expansion
\begin{equation*}
  \frac{1}{2(r+1)}(1-u)^{-1}-\frac{1}{4}+\frac{r^{2}-r-3}{24(r+1)}(1-u)+O((1-u)^{2})
\end{equation*}
for the function on the right-hand side of~\eqref{eq:moments} with $d=1$.
Then, the expansion 
\begin{multline}\label{eq:expansion-u-in-z}
  (1-u)^{-\kappa}=2^{-\kappa}(1-4z)^{-\kappa/2}+2^{-\kappa}\kappa(1-4z)^{-(\kappa-1)/2} \\ + 2^{-\kappa}\frac{\kappa(\kappa-1)}{2}(1-4z)^{-(\kappa-2)/2}+O((1-4z)^{-(\kappa-3)/2})
\end{multline}
for fixed $\kappa \in \C$ yields
\begin{equation*}
 \frac{1}{4(r+1)}(1-4z)^{-1/2}+\frac{r^{2}-r-3}{12(r+1)}(1-4z)^{1/2}+O((1-4z)^{3/2})+\text{power
   series in } (1-4z).
\end{equation*}
By singularity analysis, the $n$th coefficient, normalized by $C_{n-1}$, is asymptotically
\begin{align*}
 \E X_{n,r}=\frac{n}{r+1}-\frac{r(r-1)}{6(r+1)}+O(n^{-1})
\end{align*}
using
\begin{equation*}
C_{n-1}= 4^{n-1}\frac{1}{n^{3/2}\sqrt{\pi}\,}\Big(1+\frac{3}{8}n^{-1}+O(n^{-2})\Big).
\end{equation*}

The higher order factorial moments follow similarly by expanding the function on the
right-hand side of~\eqref{eq:moments} for general $d > 1$ around $u=1$ with the help of
SageMath, where in particular the explicit values of the derivatives of the Narayana
polynomials from Proposition~\ref{prop:narayana-derivative} are required.

Singularity analysis of the resulting expansion yields the expression given in the
statement of the theorem. Finally, note that the variance can be computed by using
\begin{equation*}
  \V X_{n,r}=\E X_{n,r}^{\underline{2}}+\E X_{n,r}-(\E X_{n,r})^{2}.
\end{equation*}
\end{proof}

\begin{theorem}
  The size $X_{n,r}$ of the tree obtained from a random plane tree with $n$
  nodes by cutting it $r$-times  is, after standardization,
  asymptotically normally distributed for $n\to\infty$ and
  fixed $r$,
  i.e.,
  \begin{equation*}
    \frac{X_{n,r}-\dfrac{n}{r+1}}{\sqrt{\dfrac{r (r + 2)}{6 (r+1)^{2}}
        n}\,}\convdistr\mathcal N(0,1).
  \end{equation*}
  To be more precise, for $x\in \R$ we have
  \[ \P\Big(\frac{X_{n,r} - n\mu}{\sqrt{\sigma^{2} n}\,} \leq x\Big) = \frac{1}{\sqrt{2\pi}\,}
    \int_{-\infty}^{x} e^{-t^{2}/2}~dt + O(n^{-1/2}), \]
  with $\mu = \frac{1}{r+1}$ and $\sigma^{2} = \frac{r (r+2)}{6 (r+1)^{2}}$ and where the
  $O$-constant depends implicitly on $r$.

As $I_{n,r-1}\overset{d}{=}X_{n,r}$, the same also holds for this
random variable.
\end{theorem}

The rest of this section is devoted to the proof of this central limit
theorem. In order to derive the fact that the number of remaining nodes after $r$
reductions is asymptotically normally distributed, we first show that the number of nodes
that are deleted after $r$ reductions is asymptotically normally distributed. Then, as the
sum of the number of remaining nodes and the number of deleted nodes is equal to the
original tree size, we obtain immediately that the number of remaining nodes has to be
asymptotically normally distributed as well.

We begin by considering the function $F_{r}\colon \mathcal{T} \to \N_{0}$ which
maps a plane tree $\tau$ to the number of nodes that are deleted when reducing the
tree $r$ times, i.e.\ the difference between the size of $\tau$ and the size of
$\rho^{r}(\tau)$. Let $\tau_{n}$ now denote a plane tree with $n$ nodes.

For the sake of convenience, we consider $F_{r}(\tau_{n})$ to be $n$ if $r$ is
larger than the maximal number of reductions that can be applied to $\tau_{n}$ before the
tree cannot be reduced further. In particular, this means that $F_{r}(\innernode) = 1$
for $r\geq 1$.

It is easy to see that the parameter $F_{r}(\tau_{n})$ is a so-called \emph{additive tree
  parameter}, meaning that
\[ F_{r}(\tau_{n}) = F_{r}(\tau_{i_{1}}) + \cdots + F_{r}(\tau_{i_{\ell}}) + f_{r}(\tau_{n})  \]
holds, where $\tau_{i_{1}}$, \dots, $\tau_{i_{\ell}}$ are the subtrees rooted at the children of
the root of $\tau_{n}$, and $f_{r}\colon \mathcal{T} \to \{0,1\}$ is a \emph{toll
  function} recursively defined by
\[ f_{r}(\tau_{n}) = \begin{cases} 1, & \text{ if } F_{r-1}(\tau_{i_{k}}) = i_{k} \text{
      for all } k = 1,\ldots, \ell,\\
    0, & \text{ otherwise, }
  \end{cases}  \]
for $r\geq 1$ and $f_{0}(\tau_{n}) = 0$.

In order to prove asymptotic normality for additive tree parameters, we can use
\cite[Theorem~2]{Wagner:2015:centr-limit}, which requires us to show that the expected
value of the toll function is exponentially decreasing in $n$. This is done in the
following lemma.

\begin{lemma}\label{lem:cut-leaves:toll-decreasing}
  The expected value of $f_{r}(\tau_{n})$ is exponentially decreasing in $n$.
\end{lemma}
\begin{remark}
  Of course, $n-F_{r}(\tau_{n})$ is also an additive parameter. However,
  the expected value of the corresponding growth function is not
  exponentially decreasing.
\end{remark}
\begin{proof}
  Define
  \begin{equation*}
    q_{n,r}=\E(f_{r}(\tau_{n}))=\P(F_{r-1}(\tau_{i_{k}})=i_{k}\text{ for  all }k=1,\ldots,\ell)=\P(F_{r}(\tau_{n})=n)
  \end{equation*}
and the corresponding generating function
\begin{equation*}
  Q_{r}(z)=\sum_{n\geq1}C_{n-1}q_{n,r}z^{n}.
\end{equation*}
Observe that $F_{r}(\tau_{n}) = n$ holds if and only if $\tau_{n}$ has height less than
$r$, as removing all leaves from a tree reduces its height by precisely one. Therefore,
the generating function $Q_{r}(z)$ is the generating function enumerating trees of height
less than $r$.

It is well-known (cf.~\cite{Bruijn-Knuth-Rice:1972}) that the generating function for
plane trees of height less than $r$ can be expressed in terms of Fibonacci
polynomials as
\[ Q_{r}(z) = \frac{z F_{r-1}(-z)}{F_{r}(-z)}.  \]
The roots of $F_{r}(-z)$ are also well-known and can be written as
$\alpha_{j,r}=(4\cos^{2}(j\pi/r))^{-1}$ for $j=1$, \dots, $\lfloor (r-1)/2\rfloor$.

Thus $Q_{r}(z)$ is a rational function and its coefficients have the form
\begin{equation*}
  C_{n-1}q_{n,r}=\sum_{j}c_{j,r}\alpha_{j,r}^{-n}
\end{equation*}
for constants $c_{j,r}$. We have $|\alpha_{j,r}| > 4$. As
\begin{equation*}
  C_{n-1}\sim\frac{4^{n-1}}{\sqrt{\pi}\,n^{3/2}},
\end{equation*}
there exists a constant $c\in(0,1)$ such that $q_{n,r}=O(c^{n})$.
\end{proof}

 Thus, by the strategy discussed above, we find that not only $F_{r}(\tau_{n})$ but also
 $X_{n,r} = n - F_{n,r}$ is asymptotically normally distributed.
\begin{remark}
  Note that the fact that $F_{1}(\tau_{n})$ is asymptotically normally distributed means that the Narayana numbers are
  asymptotically normally distributed, see for example \cite[Theorem~3.13]{Drmota:2009:random}.
\end{remark}

As sketched above, Lemma~\ref{lem:cut-leaves:toll-decreasing} allows us to
apply~\cite[Theorem 2]{Wagner:2015:centr-limit} in order to prove that $F_{r}(\tau_{n})$,
and therefore also $X_{n,r} = n - F_{r}(\tau_{n})$ is asymptotically normally
distributed. All that remains to prove is that the speed of convergence is $O(n^{-1/2})$.

We do so by noting that the proof for asymptotic normality in Wagner's theorem
is based on~\cite[Theorem 2.23]{Drmota:2009:random}, where a version of Hwang's
Quasi-Power Theorem~\cite{Hwang:1998} without quantification of the speed of convergence is used. Replacing
this argument with the multi-dimensional quantified version given
in~\cite{Heuberger-Kropf:2016:higher-dimen} then gives us the desired speed of convergence
of $O(n^{-1/2})$.

\section{Cutting Paths}\label{sec:cutting-paths}

\subsection{The Expansion Operator and Results}\label{sec:paths:results}
Let $\mathcal{P}$ denote the combinatorial class of paths, i.e.\ trees in which every node
is either a leaf or has precisely one child. The tree reduction $\rho\colon
\mathcal{T}\setminus\mathcal{P} \to \mathcal{T}$ which we will focus on in this section reduces
a tree by cutting away all paths of the tree. This operation is illustrated in
Figure~\ref{fig:paths:illustration}.
\begin{figure}[ht]
  \centering
      \[
    \begin{tikzpicture}[thick, scale=0.75, baseline={([yshift=-0.5em]current bounding
      box.center)}]
      \node[draw, circle] {}
      child[gray, dashed] {node[draw, circle] {} child {node[draw, circle] {} child{node[draw,
            rectangle] {}}}}
      child {node[draw, circle] {}
        child[gray, dashed] {node[draw, rectangle] {}}
        child[gray, dashed] {node[draw, rectangle] {}}
      }
      child[gray, dashed] {node[draw, circle] {} child {node[draw, rectangle] {}}};
    \end{tikzpicture}
    \quad \mapsto \quad
    \begin{tikzpicture}[thick, scale=0.75, baseline={([yshift=-0.5em]current bounding
      box.center)}]
      \node[draw, circle] {} child {node[draw, rectangle] {}};
    \end{tikzpicture}
    \]
  \caption{Illustration of the ``cutting paths''-operator $\rho$}
  \label{fig:paths:illustration}
\end{figure}
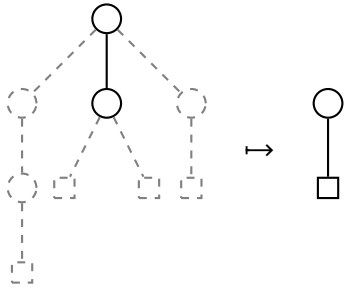

Analogously to our approach in Section~\ref{sec:cut-leaves:expansion}, we first determine
the corresponding expansion operator $\Phi$. In order to do so, we need the generating
function for the family of paths $\mathcal{P}$, which is given by $P = P(z,t) =
\frac{t}{1-z}$. For the sake of readability, we omit the arguments of $P$.

\begin{proposition}\label{prop:paths:expansion}
  Let $\mathcal{F}\subseteq \mathcal{T}$ be a family of plane trees with bivariate
  generating function $f(z,t)$, where $z$ marks inner nodes and $t$ marks leaves. Then the
  generating function for $\rho^{-1}(\mathcal{F})$, the family of trees whose reduction is
  in $\mathcal{F}$, is given by
  \begin{equation}\label{eq:paths:expansion}
    \Phi(f(z,t)) = (1-P) f\Big(\frac{z}{(1-P)^{2}}, \frac{zP^{2}}{(1-P)^{2}}\Big).
  \end{equation}
\end{proposition}
\begin{proof}
  The fact that $\Phi$ is a linear operator is obvious from a combinatorial point of view,
  meaning that we may concentrate on some tree $\tau$ with $n$ inner nodes and $k$ leaves,
  represented by $z^{n}t^{k}$.

  We follow the proof of Proposition~\ref{prop:cut-leaves:expansion} and observe that all
  possible tree expansions of $\tau$ can be obtained by the following
  operations: the leaves of $\tau$ are expanded by appending a sequence of at least two paths to each
  of them. Note that appending a single path to a leaf is not allowed, because this would just
  extend the path ending in that leaf, which causes ambiguity. Then, the inner nodes are
  expanded as well by appending (possibly empty) sequences of paths to the $2n+k-1$
  available positions between, before, and after their children.

  Translating this expansion to the language of generating functions yields
  \[ \Phi(z^{n}t^{k}) = z^{n} \Big(\frac{zP^{2}}{1-P}\Big)^{k}
    \frac{1}{(1-P)^{2n+k-1}},  \]
  which proves~\eqref{eq:paths:expansion}.
\end{proof}
\begin{corollary}\label{cor:paths:functional-equation}
  The generating function for plane trees $T(z,t)$ satisfies the functional
  equation
  \begin{equation}\label{eq:paths:functional-equation}
    T(z,t) = P + \Phi(T(z,t)).
  \end{equation}
\end{corollary}
\begin{proof}
  Surjectivity of $\rho$ implies $\rho^{-1}(\mathcal{T}) =
  \mathcal{T}\setminus\mathcal{P}$, which proves the statement after translating this into
  the language of generating functions with the help of $\Phi$.
\end{proof}

In the following proposition, we determine the generating function $G_{r}(z,v_{I},v_{L})$ measuring
the effect of applying the path reduction $r$ times on the size of the tree. Most
interestingly, we will see that the path connection is in fact strongly related to the leaf
reduction from the previous section.
\begin{proposition}\label{prop:paths:gf}
  The trivariate generating function $G_{r}(z,v_{I},v_{L}) = G_{r}^{\operatorname{P}}(z,v_{I},v_{L})$ enumerating plane trees whose
  paths can be cut at least $r$-times, where $z$ marks the tree size and $v_{I}$ and $v_{L}$ mark the
  number of inner nodes and leaves of the $r$-fold cut tree, respectively, is given by
  \begin{align*}
    G_{r}(z,v_{I},v_{L})&=\Phi^{r}(T(zv_{I},tv_{L}))\vert_{t=z} \\ &=
    \frac{1-u^{2^{r+1}}}{(1-u^{2^{r+1}-1})(1+u)}T\Bigg(\frac{u(1-u^{2^{r+1}-1})^{2}}{(1-u^{2^{r+1}})^{2}}v_{I},\frac{u^{2^{r+1}-1}(1-u)^{2}}{(1-u^{2^{r+1}})^{2}}v_{L}\bigg),
  \end{align*}
  where $z = u/(1+u)^{2}$.
\end{proposition}
\begin{proof}
  By the same reasoning as in the proof of Proposition~\ref{prop:cut-leaves:gf}, the
  generating function we are interested in is $G_{r}(z,v_{I},v_{L}) = \Phi^{r}(T(zv_{I}, tv_{L}))|_{t=z}$,
  meaning that we want to study the iterated application of $\Phi$. To do so, we consider
  the strongly related operator
  \[ \Psi(f(z,t)) \coloneqq f\Big(\frac{z}{(1-P)^{2}},\frac{zP^{2}}{(1-P)^{2}}\Big). \]
  The relation
  \[ \Phi^{r}(f(z,t)) = \Psi^{r}(f(z,t)) \prod_{j=0}^{r-1}(1 - \Psi^{j}(P))  \]
  can be proved easily by induction and enables us to determine the behavior of $\Phi$ via
  $\Psi$.

  First of all, for $r\geq 0$ and $r\geq 1$, the relations
  \[ \Psi^{r}(z) = \frac{z}{\prod_{j=0}^{r-1} (1 - \Psi^{j}(P))^{2}} \quad\text{ and
    }\quad \Psi^{r}(P) = \frac{z (\Psi^{r-1}(P))^{2}}{\prod_{j=0}^{r-1} (1 - \Psi^{j}(P))^{2} -
      z} \]
  can be proved easily by induction, respectively. Also observe that we can write $\Psi^{r}(t) =
  \Psi^{r}(z) \Psi^{r-1}(P)^{2}$. Now let $f_{r} = \Psi^{r}(z)|_{t=z}$, $g_{r} =
  \Psi^{r}(t)|_{t=z}$, and $h_{r} = \Psi^{r}(P)|_{t=z}$. With the help of the identity
  $\prod_{j=0}^{r} (1 + u^{2^{j}}) = \frac{1 - u^{2^{r+1}}}{1-u}$ we are able to prove the
  explicit formula
  \begin{equation}\label{eq:paths:psi-P}
    h_{r} = \frac{u^{2^{r+1} - 1} (1-u)}{1 - u^{2^{r+2} - 1}} = \frac{z^{2^{r+1} -
        1}}{F_{2^{r+2}-1}(-z)},
  \end{equation}
  where $z = u/(1+u)^{2}$ and the second equation is a consequence
  of~\eqref{eq:fib-poly-u}. Using~\eqref{eq:paths:psi-P}, we immediately find
  \[ f_{r} = \frac{u (1 - u^{2^{r+1} -1})^{2}}{(1 - u^{2^{r+1}})^{2}} = \frac{z
      F_{2^{r+1}-1}(-z)^{2}}{F_{2^{r+1}}(-z)^{2}} \quad\text{ and }\quad g_{r} =
    \frac{u^{2^{r+1}-1} (1-u)^{2}}{(1 - u^{2^{r+1}})^{2}} = \frac{z^{2^{r+1} -
        1}}{F_{2^{r+1}}(-z)^{2}}.  \]
  Putting everything together yields
  \begin{align*} \Phi^{r}(z^{n}t^{k})|_{t=z} &= \frac{z^{n + (2^{r+1}-1)k}
      F_{2^{r+1}-1}(-z)^{2n-1}}{F_{2^{r+1}}(-z)^{2n+2k-1}} \\ &= \frac{1 -
      u^{2^{r+1}}}{(1-u^{2^{r+1} -1})(1+u)} \Big(\frac{u (1 - u^{2^{r+1}-1})^{2}}{(1 -
      u^{2^{r+1}})^{2}}\Big)^{n}  \Big(\frac{u^{2^{r+1}-1}
      (1-u)^{2}}{(1-u^{2^{r+1}})^{2}}\Big)^{k},
    \end{align*}
  which directly implies the statement.
\end{proof}

The following result shows that there is an intimate connection between the ``cutting
leaves''-reduction from Section~\ref{sec:cutting-leaves} and the ``cutting
paths''-reduction, as can be seen after comparing the statement of
Proposition~\ref{prop:cut-leaves:gf} with the statement of Proposition~\ref{prop:paths:gf}.
\begin{corollary}\label{cor:paths:same-gf}
  The generating function $G_{r}(z,v_{I},v_{L}) = G_{r}^{\operatorname{P}}(z,v_{I},v_{L})$ measuring the change in size
  after cutting away all paths from plane trees $r$ times is equal to the generating
  function $G_{2^{r+1}-2}^{\operatorname{L}}(z,v_{I},v_{L})$ measuring the change in size after
  cutting away all leaves from plane trees $2^{r+1}-2$ times.
\end{corollary}

This connection is now especially important for the analysis of the random variable
$X_{n,r} = X_{n,r}^{\operatorname{P}}$ modeling the number of nodes that are left after
reducing a random tree $\tau$ with $n$ nodes $r$ times by removing all paths. In fact, it
follows that
\[ X_{n,r}^{\operatorname{P}} \overset{d}{=} X_{n,2^{r+1}-2}^{\operatorname{L}},  \]
meaning that the asymptotic analysis of the factorial moments of
$X_{n,r}^{\operatorname{P}}$ as well as the limiting distribution follow directly from the
corresponding results in Section~\ref{sec:cut-leaves:analysis}.

\begin{theorem}\label{thm:cutting-paths-clt}
   Let $r\in\N_{0}$ be fixed and consider $n\to\infty$. Then expectation and variance of the
   random variable $X_{n,r} = X_{n,r}^{\operatorname{P}}$ can be expressed as
  \begin{equation}\label{eq:cut-paths:exp}
    \E X_{n,r} = \frac{n}{2^{r+1}-1} - \frac{(2^{r}-1)(2^{r+1}-3)}{3(2^{r+1}-1)} + O(n^{-1}),
  \end{equation}
  and
  \begin{equation}\label{eq:cut-paths:var}
    \V X_{n,r} = \frac{ 2^{r+1}(2^{r}-1)}{3 (2^{r+1}-1)^{2}} n + O(1).
  \end{equation}

The factorial moments are asymptotically given by
\begin{align*}
  \E X_{n,r}^{\underline{d}}&=\frac{n^{d}}{(2^{r+1}-1)^{d}}\\*
&\quad+ \frac{d}{12 (2^{r+1} - 1)^{d}}(4^{r+1} d - 2^{r+4} d - 3\cdot 4^{r+1} + 9\cdot 2^{r+2} + 6d - 18) \\* &\quad+O(n^{d-3/2}).
\end{align*}

Furthermore, $X_{n,r} = X_{n,r}^{\operatorname{P}}$ is asymptotically normally distributed,
i.e., for $x\in \R$ we have
\begin{equation*}
  \P\Bigg(\frac{X_{n,r}- \mu n}{\sqrt{\sigma^{2} n}\,} \leq x\Bigg) = \frac{1}{\sqrt{2\pi}\,} \int_{-\infty}^{x} e^{-t^{2}/2}~dt + O(n^{-1/2})
\end{equation*}
for $\mu = \frac{1}{2^{r+1} - 1}$ and $\sigma^{2} = \frac{2^{r+1} (2^{r} - 1)}{3
  (2^{r+1} - 1)^{2}}$. All $O$-constants in this theorem depend implicitly on $r$.
\end{theorem}

\subsection{Total number of paths}\label{sec:paths:totalnumber}

In the context of this reduction it is interesting to investigate the total number of
paths needed to construct a given tree.
To determine this parameter we can reduce the tree repeatedly
and count the number of leaves. The sum of the number of leaves over all reduction steps
is equal to the number of paths, which follows from the observation that leaves mark the
endpoints of all paths.

Formally, given the random variables $P_{n,r}$ counting the number of leaves in the $r$th
reduction of a tree of size $n$, we want to analyze the random variable $P_{n} \coloneqq
\sum_{r\geq 0} P_{n,r}$.

\begin{proposition}
  The expected number of paths needed to construct a uniformly random tree of size $n$ satisfies
  \begin{equation}
    \label{eq:paths:total-expansion}
    \E P_{n} = \frac{1}{C_{n-1}} [z^{n}] \frac{1-u}{1+u} \sum_{r\geq 1}
    \frac{u^{2^{r}}}{(1 - u^{2^{r}})(1 - u^{2^{r} - 1})},
  \end{equation}
  where $z = u/(1+u)^{2}$.
\end{proposition}
\begin{proof}
  As a consequence of Proposition~\ref{prop:paths:gf}, the bivariate generating function
  enumerating plane trees where $z$ marks tree size and $v$ marks the number of
  leaves after $r$ path reductions can be written as
  \[ \frac{1 - u^{2^{r+1}}}{(1-u^{2^{r+1} - 1})(1+u)} T\Big(\frac{u (1 - u^{2^{r+1} -
      1})^{2}}{(1 - u^{2^{r+1}})^{2}}, \frac{u^{2^{r+1} - 1}
    (1-u)^{2}}{(1-u^{2^{r+1}})^{2}} v\Big).  \]
  By differentiating this generating function once with respect to $v$ and setting $v = 1$
  afterwards, we obtain an expression where $C_{n-1} \E P_{n,r}$ can be extracted as the
  coefficient of $z^{n}$. By~\eqref{eq:moments-fringe} with $d=1$ and $r$ replaced by
  $2^{r+1}-2$, we have
  \[ \E P_{n,r} = \frac{1}{C_{n-1}} [z^{n}] \frac{1-u}{1+u} \frac{u^{2^{r+1}}}{(1 -
    u^{2^{r+1}}) (1 - u^{2^{r+1} - 1})}.  \]
  Summation over $r\geq 0$ and shifting the index of summation by one completes the
  proof.
\end{proof}

Our strategy for determining an asymptotic expansion for $\E P_{n}$ as given
in~\eqref{eq:paths:total-expansion} is based on the Mellin transform.

\begin{theorem}\label{thm:paths:total}
  For $n\to\infty$, the expected number of paths required to construct a uniformly random
  tree of size $n$ is given by the asymptotic expansion
  \begin{equation}
    \label{eq:paths:total-asymptotic}
    \E P_{n} = (\alpha - 1) n + \frac{1}{6} \log_{4}n - \frac{\gamma +
      4(\alpha-1)\log 2 + \log 2 + 24\zeta'(-1) + 2}{12\log 2} + \delta(\log_{4} n) + O(n^{-1/4}),
  \end{equation}
  where
  \begin{equation}\label{eq:paths:total-fluctuation}
    \delta(x) \coloneqq \frac{1}{\log 2}\sum_{k\in \Z\setminus\{0\}} (-1 + \chi_{k})
    \Gamma(\chi_{k}/2) \zeta(-1 + \chi_{k}) e^{2k\pi i x}
  \end{equation}
  with $\chi_{k} = \frac{2k\pi i}{\log 2}$ is a fluctuation with mean $0$ and $\alpha \coloneqq \sum_{k\geq 1} 1/(2^{k} - 1) \approx
  1.606695$, $\gamma$ is the Euler--Mascheroni constant and $\zeta$ is the Riemann zeta function.
\end{theorem}
\begin{remark}
  The constant $\alpha$ appears in the asymptotic analysis of digital search
  trees (see e.g.~\cite{Kirschenhofer-Prodinger:1988:digital-search-trees}).
\end{remark}
\begin{proof}
  In order to obtain an asymptotic expansion from~\eqref{eq:paths:total-expansion}, we
  rewrite
  \[ P(z) = \frac{1-u}{1+u} \sum_{r\geq 1} \frac{u^{2^{r}}}{(1-u^{2^{r}})(1-u^{2^{r}-1})} =
    \frac{u}{1+u}\sum_{r\geq 1} \Big(\frac{u^{2^{r}-1}}{1 - u^{2^{r}-1}} - \frac{u^{2^{r}}}{1 -
      u^{2^{r}}}\Big)  \]
  where $z = u/(1+u)^{2}$.
  The main task to obtain an asymptotic expansion of $P(z)$ is to provide a precise analysis of
  this sum, which we carry out via the Mellin transform. We consider the function
  \[ f(t) \coloneqq \sum_{r\geq 1} \frac{e^{-(2^{r}-1) t}}{1 - e^{-(2^{r}-1) t}} - \sum_{r\geq 1} \frac{e^{-2^{r}t}}{1 -
      e^{-2^{r}t}}, \]
  obtained from substituting $u = e^{-t}$ in the sum above. With
  \[ A(s) \coloneqq \sum_{r\geq 1}\frac{1}{2^{r s}}((1 - 2^{-r})^{-s} - 1)  = \sum_{\ell\geq 1} \binom{\ell + s - 1}{\ell} \frac{1}{2^{s + \ell} - 1} \]
  we find that the corresponding Mellin transform of this difference of harmonic sums is
  given by
  \[ f^{*}(s) = \Gamma(s) \zeta(s) A(s) \]
  with fundamental strip $\langle 1,\infty \rangle$. In order for the inversion formula to
  be valid, we need to show that $f^{*}(s)$ decays sufficiently fast along vertical lines
  in the complex plane. While $\Gamma(s)$ and $\zeta(s)$ are well-known to decay
  exponentially and grow polynomially along vertical lines, respectively, the Dirichlet
  series $A(s)$ has to be investigated in more detail.

  We want to estimate the summands in
  \[ A(s) - \frac{s}{2^{s+1} - 1} = \sum_{r\geq 1}\frac{1}{2^{rs}}\Big((1-2^{-r})^{-s} - 1 -
    \frac{s}{2^{r}}\Big).  \]
  To do so, we consider $g(x) = (1-x)^{-s}$ as a function of a real variable. By means of
  the integral form of the Taylor approximation error we find
  \begin{align*}
    \abs{g(2^{-r}) - g(0) - g'(0)\cdot 2^{-r}}
    & = \abs[\Big]{\int_{0}^{2^{-r}} s(s+1) (1-t)^{-s-2} (2^{-r}-t)~dt}\\
    & \leq \abs{s} \abs{s+1} 2^{-r} \int_{0}^{2^{-r}} \abs{1-t}^{-\Re s - 2}~dt \\
    & \leq \abs{s} \abs{s+1} 2^{-2r} (1-2^{-r})^{-\Re s - 2},
  \end{align*}
  where the last inequality is valid under the assumption that $\Re s > -2$. Using this
  estimate, we find
  \begin{align*}
    \abs{A(s)} & \leq \abs[\Big]{A(s) - \frac{s}{2^{s+1} - 1}} +
                 \abs[\Big]{\frac{s}{2^{s+1} - 1}}\\
               & \leq \abs[\Big]{\frac{s}{2^{s+1} - 1}} + \abs{s}\abs{s+1} \sum_{r\geq 1}
                 \frac{1}{(2^{r} - 1)^{\Re s + 2}},
  \end{align*}
  where the sum converges for $\Re s > -2$. Therefore, $A(s)$ has polynomial growth in $\Im s$ for
  $\Re s > -2$ and $\Im s = \frac{2\pi i}{\log 2} \big(k + \frac{1}{2}\big)$, where
  $k\in\Z$ and $\abs{k} \to \infty$, as well as on vertical lines with $\Re s>-2$ and $\Re s\neq -1$. This implies that $f^{*}(s)$ decays sufficiently
  fast, and thus the inversion formula states 
  \begin{equation}\label{eq:totalpaths:inversemellin}
    f(t) = \frac{1}{2\pi i} \int_{2-i\infty}^{2+i\infty} \Gamma(s)\zeta(s)A(s)
    t^{-s}~ds,
  \end{equation}
  which is valid for real, positive $t \to 0$ (and thus $u\to 1^{-}$ and $z \to
  (1/4)^{-}$, as we have $z = u/(1+u)^{2}$ and $u = e^{-t}$). In order to extract the
  coefficient growth (in terms of $z$) with the help of singularity analysis, we require
  analyticity in a larger region (cf.~\cite{Flajolet-Odlyzko:1990:singul}), e.g. in a
  complex punctured neighborhood of $1/4$ with\footnote{Note that the bound $2\pi/5$ is
    somewhat arbitrary: the argument just needs to be less than $\pi/2$.} $\abs{\arg(z - 1/4)} >
  2\pi/5$.

  Substituting back $t$ for $z$, we find
  \[ t = -\log\Big(\frac{1 - \sqrt{1 - 4z}\,}{2z} - 1\Big) = 2\sqrt{1 - 4z}\, + \frac{2}{3} (1
    - 4z)^{3/2} + O((1-4z)^{5/2}), \]
  which implies
  \[ \abs{\arg t} = \frac{1}{2} \abs{\arg(1 - 4z)} + o(1)  \]
  such that we have the bound $\abs{\arg t} < 2\pi/5$ for $t \to 0$, given that the
  restriction on the argument in terms of $z$ is satisfied.

  With the help of our estimates on $f^{*}(s)$ that we discussed above, we find that
  \begin{equation}\label{eq:totalpaths:mellin-estimate}
  \abs{f^{*}(s)t^{-s}} = O\Big(\abs{\Im(t)}^{4} \abs{t}^{-\Re(s)}
    \exp\Big(-\frac{\pi}{10} \abs{\Im(s)}\Big)\Big)
  \end{equation}
  for $-3/2 \leq \Re s \leq 2$ and $\Im s = \frac{2\pi i}{\log 2} \big(k +
  \frac{1}{2}\big)$, where $k\in\Z$ and $\abs{k} \to \infty$. This is a consequence of
  combining the quantified growth of $\Gamma(s)$ (see~\DLMF{5.11}{3})
  and the growth of $\zeta(s)$ (see~\cite[13.51]{Whittaker-Watson:1996}) with the facts
  that $A(s)$ is of order $O(\Im(s)^{2})$ and $\frac{s}{2^{s+1}-1}$ is of order
  $O(\Im(s))$ for $s$ taking values in the specified region.
  
  We can evaluate~\eqref{eq:totalpaths:inversemellin} by shifting the line of integration
  from $\Re(s) = 2$ to $\Re(s) = -3/2$ and collecting the residues of the poles we
  cross. This yields
  \[ f(t) =  \sum_{p\in P} \Res_{s=p}(f^{*}(s) t^{-s}) + \frac{1}{2\pi i}
    \int_{-3/2-i\infty}^{-3/2 + i\infty} f^{*}(s)t^{-s}~ds,  \]
  where $P = \{-1,1\} \cup \{-1 + \chi_{k}\mid k\in \Z\setminus \{0\}\}$.
  For the error term we use the estimate above and find
  \[ \frac{1}{2\pi i} \int_{-3/2 - i\infty}^{-3/2 + i\infty} f^{*}(s)t^{-s}~ds =
    O(\abs{t}^{3/2}). \]

  Evaluating the residues yields
  \begin{multline*}
    f(t) = A(1) t^{-1} + \frac{1}{12\log 2} t\log t + \frac{\log 2 + 2\gamma + 24\zeta'(-1)}{24\log 2} t \\ + \sum_{k\in \Z\setminus\{0\}}
    \frac{1}{\log 2} \Gamma(\chi_{k}) \zeta(-1 + \chi_{k}) t^{1-\chi_{k}}
    + O(\abs{t}^{3/2}).
  \end{multline*}
  Note that with $\alpha \coloneqq \sum_{k\geq 1} 1/(2^{k} - 1)$, we have $A(1) = \alpha - 1$.

  When substituting back in order to obtain an expansion in terms of $z\to 1/4$, we have
  to carefully check that the error terms within the sum of the residues at $\chi_{k}$ for
  $k\in \Z\setminus \{0\}$ can still be controlled. Considering that for some exponent
  $\kappa$, we have the expansion
  \[ t^{-\kappa} = (1-4z)^{-\kappa/2} (1 + O(1-4z))^{-\kappa/2},  \]
  and thus
  \begin{align*}
  \abs{(1+O(1-4z))^{-\kappa/2} - 1} &= \abs[\Big]{\exp\Big(-\frac{\kappa}{2}
                                        \log(1+O(1-4z))\Big) - 1}\\
  & \leq
  \abs[\Big]{\frac{\kappa}{2}} \abs{\log(1+O(1-4z))} \exp\Big(\abs[\Big]{\frac{\kappa}{2}}
  \abs{\log(1+O(1-4z))}\Big)\\
  & = \abs[\Big]{\frac{\kappa}{2}} O(1-4z) \exp\Big(\abs[\Big]{\frac{\kappa}{2}} O(1-4z)\Big).
  \end{align*}
  Setting $\kappa = -1+\chi_{k}$ shows that the errors that we sum are of order $O(\abs{k}(1-4z)\exp(\abs{k} O(1-4z)))$.
  Choosing $z$ sufficiently close to $1/4$ ensures that the
  exponential growth is  negligible compared to the exponential decay proved
  in~\eqref{eq:totalpaths:mellin-estimate}.

  Finally, it is easy to see that the factor $\frac{u}{1+u}$ can be rewritten as $\frac{1 - \sqrt{1
      - 4z}\,}{2}$. Multiplying our expansion of $f(t)$ with this factor and substituting back
  yields the expansion
  \begin{multline*}
    P(z) = \frac{\alpha - 1}{4} (1-4z)^{-1/2} - \frac{\alpha - 1}{4} - \frac{1}{24\log
      2} (1-4z)^{1/2} \log(1-4z) \\ + \frac{2\gamma - 2\alpha\log 2 + 5\log 2 +
      24\zeta'(1)}{24\log 2} (1-4z)^{1/2} \\+ \frac{1}{\log 2} \sum_{k\in\Z\setminus\{0\}}
    \Gamma(\chi_{k}) \zeta(-1+\chi_{k}) (1-4z)^{(1 - \chi_{k})/2}+ O((1-4z)^{3/4}).
  \end{multline*}
  Applying singularity analysis, normalizing the result by $C_{n-1}$, and rewriting the
  coefficients of the contributions from the poles at $-1 + \chi_{k}$ via the duplication
  formula for the Gamma function (cf.~\DLMF{5.5}{5}) then proves the asymptotic expansion for
  $\E P_{n}$.
\end{proof}

\section{Cutting Old Leaves}
\label{sec:cutting-old-leaves}

\subsection{Preliminaries}\label{sec:old-leaves:preliminaries}
In this section we consider a slightly more complex reduction: instead of removing all
leaves, we just remove all leftmost
leaves. Following~\cite{Chen-Deutsch-Elizalde:2006:old-leaves}, we call a leaf that is a
leftmost child an \emph{old leaf}.

In order to describe the corresponding expansion in the language of generating functions,
we need to change our underlying combinatorial model of trees in a way that specifically
marks old leaves.

Let $\mathcal{L}$ be the combinatorial class of plane trees where $\blacksquare$
marks old leaves and $\innernode$ marks all
nodes that are neither old leaves nor parents thereof. Now, as a first step we determine
the bivariate generating function $L(z,w)$ of $\mathcal{L}$.

\begin{proposition}\label{proposition:gf-L-explicit}
  The generating function $L(z,w)$ enumerating plane trees
  with respect to old leaves $\blacksquare$
  (marked by the variable $w$) and all nodes $\innernode$ that are neither old leaves nor parents
  thereof (marked by $z$) is given by
  \begin{equation}
    \label{eq:gf-old-leaves}
    L(z,w)=\frac{1 - \sqrt{1 - 4z - 4w + 4z^{2}}\,}{2}.
  \end{equation}
  For $n\geq 2$ there are $C_{k-1}\binom{n-2}{n-2k} 2^{n-2k}$ plane trees of
  size $n$ (meaning $n$ nodes overall) with $k$ old leaves.
\end{proposition}
For example in Figure~\ref{fig:old-leaves:illustration}, the original tree corresponds to $z^{3}w^{3}$
because it has three old leaves (dashed nodes) and three nodes which are neither old
leaves nor parents of old leaves.
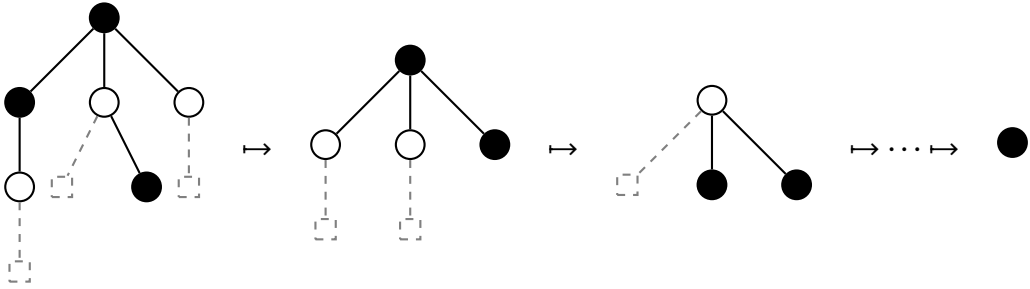
\begin{figure}[ht]
  \centering
        \[
    \begin{tikzpicture}[thick, scale=0.75, baseline={([yshift=-0.5em]current bounding
      box.center)}]
      \node[draw, circle, fill] {}
      child[] {node[draw, circle, fill] {} child {node[draw, circle] {} child[gray, dashed]{node[draw,
            rectangle] {}}}}
      child {node[draw, circle] {}
        child[gray, dashed] {node[draw, rectangle] {}}
        child[] {node[draw, circle, fill] {}}
      }
      child[] {node[draw, circle] {} child[gray, dashed] {node[draw, rectangle] {}}};
    \end{tikzpicture}
    \quad \mapsto \quad
    \begin{tikzpicture}[thick, scale=0.75, baseline={([yshift=-0.5em]current bounding
      box.center)}]
      \node[draw, circle, fill] {}
      child[] {node[draw, circle] {} child[gray, dashed] {node[draw, rectangle] {} }}
      child {node[draw, circle] {}
        child[gray, dashed] {node[draw, rectangle] {}}
      }
      child[] {node[draw, circle, fill] {} };
    \end{tikzpicture}
        \quad \mapsto \quad
    \begin{tikzpicture}[thick, scale=0.75, baseline={([yshift=-0.5em]current bounding
      box.center)}]
      \node[draw, circle] {}
      child[gray, dashed] {node[draw, rectangle] {}}
      child {node[draw, circle, fill] {}}
      child {node[draw, circle, fill] {}
      };
    \end{tikzpicture}
            \quad \mapsto \dots \mapsto \quad
    \begin{tikzpicture}[thick, scale=0.75, baseline={([yshift=-0.5em]current bounding
      box.center)}]
      \node[draw, circle, fill] {
      };
    \end{tikzpicture}
    \]
  \caption{Illustration of the ``cutting old leaves''-operator $\rho$}
  \label{fig:old-leaves:illustration}
\end{figure}

\begin{proof}
  We consider the symbolic equation describing the combinatorial class
   $\mathcal{L}$ of plane trees with respect to old leaves, which is illustrated
  in Figure~\ref{fig:old-leaves:symbolic}.
  \begin{figure}[ht]
    \centering
    \begin{tikzpicture}
      \node (add) {$\mathcal{L}\quad=\quad\tikz{\node[draw, circle, fill, inner sep=2.5pt] {};}\quad+\quad\sum\limits_{n\geq0}$};
      \node[right of=add, circle, fill, outer sep=0pt, inner sep=0pt,  xshift=7em, yshift=2em] (right-V) {};
      \node[below of=right-V, yshift=-1.5em, xshift=-4em, draw, fill] (1) {};
      \node[below of=right-V, yshift=-1.5em, xshift=-2em] (2) {$\mathcal{L}$};
      \node[below of=right-V, yshift=-1.5em, xshift=0em]  (3) {$\mathcal{L}$};
      \node[below of=right-V, yshift=-1.5em, xshift=2em, gray] (4) {$\cdots$};
      \node[below of=right-V, yshift=-1.5em, xshift=4em] (5) {$\mathcal{L}$};
      \draw (right-V) -- (1) (right-V) -- (2) (right-V) -- (3) (right-V) -- (5);
      \draw[dotted, gray] (right-V) -- (4);
      \draw [thick, decoration={brace, mirror, raise=1em},
             decorate] (2.center) to node (h) {} (5.center);
      \node [below of=h] {$n$};
      \node[right of=add, xshift=13.5em] {$+\quad \sum\limits_{n\geq0}$};
      \node[right of=right-V, circle, fill=black, inner sep=3pt, xshift=11.5em] (right-V2) {};
      \node[below of=right-V2, yshift=-1.5em, xshift=-5em] (12) {$\mathcal{L} - \tikz{\node[draw, circle, fill, inner sep=2.5pt] {};}$};
      \node[below of=right-V2, yshift=-1.5em, xshift=-2em] (22) {$\mathcal{L}$};
      \node[below of=right-V2, yshift=-1.5em, xshift=0em]  (32) {$\mathcal{L}$};
      \node[below of=right-V2, yshift=-1.5em, xshift=2em, gray] (42) {$\cdots$};
      \node[below of=right-V2, yshift=-1.5em, xshift=4em] (52) {$\mathcal{L}$};
      \draw (right-V2) -- (12) (right-V2) -- (22) (right-V2) -- (32) (right-V2) -- (52);
      \draw[dotted, gray] (right-V2) -- (42);
            \draw [thick, decoration={brace, mirror, raise=1em},
             decorate] (22.center) to node (h2) {} (52.center);
      \node [below of=h2] {$n$};
    \end{tikzpicture}
    \caption{Symbolic equation for plane trees w.r.t.\ old
      leaves}
    \label{fig:old-leaves:symbolic}
  \end{figure}
  The functional equation that can be derived from the symbolic equation by marking
  $\blacksquare$ with $w$ and $\innernode$ with $z$ is
  \begin{equation}\label{eq:gf-old-leaves:func-eq} 
    L(z,w) = z + \frac{w + z (L(z,w) - z)}{1 - L(z,w)}.
  \end{equation}
  Solving this equation and choosing the correct branch of the root
  yields~\eqref{eq:gf-old-leaves}.
 
    To extract coefficients of $L(z, w)$, we rewrite it as
    \begin{align}\label{eq:L-sqrt-z}
      L(z, w) &= \frac12\biggl(1- (1-2z)\sqrt{1-\frac{4w}{(1-2z)^2}}\,\biggr)
      = \frac12\biggl(1-\sum_{k\ge 0} \binom{1/2}{k}\frac{(-1)^k4^kw^k}{(1-2z)^{2k-1}}\biggr)\\
      &= z + \sum_{k\ge 1}C_{k-1}\frac{w^k}{(1-2z)^{2k-1}}
      = z + \sum_{\substack{k\ge 1\\n\ge 0}}C_{k-1}\binom{n+2k-2}{n}2^n w^k
      z^{n}.\label{eq:L-full-expansion}
    \end{align}
\end{proof}

As we will see in the next section, the polynomials defined below will play a similar role
for the ``old leaves''-reduction as the Fibonacci polynomials played for the ``leaves''-
and ``paths''-reduction.
\begin{definition}
  The polynomials $B_{r}(z)$ are the generating functions of binary trees w.r.t.\ the
  number of internal nodes of height $\leq r$ satisfying
\begin{equation}\label{eq:binary-height-poly}
  B_{r}(z)=1+zB_{r-1}(z)^{2}
\end{equation}
for $r\geq 1$ and $B_{0}(z)=1$.
\end{definition}

\subsection{The Expansion Operator and Asymptotic Results}\label{sec:old-leaves:expansion}
As described in the previous section, we now concentrate on the reduction $\rho\colon
\mathcal{L} \to \mathcal{L}$, which removes all old leaves from a
tree. Note that $\rho(\innernode) = \innernode$, as
the root itself is not an old leaf.
We begin our analysis of this reduction by determining the expansion operator $\Phi$.

\begin{proposition}\label{prop:old-leaves:expansion}
  Let $\mathcal{F}\subseteq \mathcal{L}$ be a family of plane trees with bivariate
  generating function $f(z,w)$, where $z$ marks nodes that are neither old leaves
  nor parents thereof and $w$ marks old leaves. Then the generating function
  for $\rho^{-1}(\mathcal{F})$, the family of trees whose reduction is in $\mathcal{F}$,
  is given by
  \begin{equation}\label{eq:old-leaves:expansion}
    \Phi(f(z,w)) = f(z+w, (2z+w)w).
  \end{equation}
\end{proposition}
\begin{proof}
  Linearity of $\Phi$ is obvious from the combinatorial interpretation, meaning that we
  can focus on the expansion of any tree represented by $z^{n}w^{k}$, i.e.\ a tree with $n$
  nodes that are neither old leaves nor parents thereof and $k$ old leaves.

  Figure~\ref{fig:old-leaves:w-expansion} illustrates all three possibilities to expand an old leaf $\blacksquare$:
  \begin{itemize}
  \item[--] appending an old leaf to the parent of $\blacksquare$, which turns the original
    old leaf into $\innernode$,
  \item[--] appending an old leaf to $\blacksquare$ itself, which turns the parent into
    $\innernode$,
  \item[--] appending old leaves both to $\blacksquare$ and its parent.
  \end{itemize}
  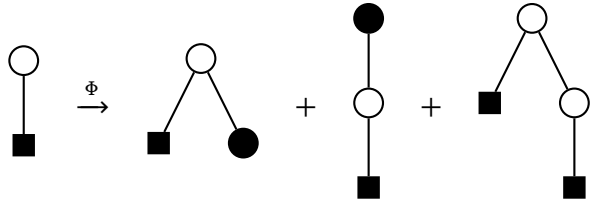
\begin{figure}[ht]
    \centering
    \[
      \begin{tikzpicture}[thick, scale=0.75, baseline={([yshift=-0.5em]current bounding
          box.center)}]
        \node[draw, circle] {} child {node[draw, rectangle, fill] {}};
      \end{tikzpicture}
      \quad \xrightarrow{\Phi} \quad
      \begin{tikzpicture}[thick, scale=0.75, baseline={([yshift=-0.5em]current bounding
          box.center)}]
        \node[draw, circle] {}
          child {node[draw, rectangle, fill] {}}
          child {node[draw, circle, fill] {}};
      \end{tikzpicture}
      \quad + \quad
      \begin{tikzpicture}[thick, scale=0.75, baseline={([yshift=-0.5em]current bounding
          box.center)}]
        \node[draw, circle, fill] {} child {node[draw, circle] {} child {node[draw,
            rectangle, fill] {}}};
      \end{tikzpicture}
      \quad + \quad
      \begin{tikzpicture}[thick, scale=0.75, baseline={([yshift=-0.5em]current bounding
          box.center)}]
        \node[draw, circle] {}
          child {node[draw, rectangle, fill] {}}
          child {node[draw, circle] {} child {node[draw, rectangle, fill] {}}};
      \end{tikzpicture}
    \]
    \caption{All possible expansions of an old leaf}
    \label{fig:old-leaves:w-expansion}
  \end{figure}
  In terms of generating functions, this means that $w$ is substituted by $2zw + w^{2}$.

  Furthermore, the nodes represented by $\innernode$ can optionally be expanded by attaching an old
  leaf to them, otherwise they stay as they are. This option corresponds to the substitution $z \mapsto z+w$.

  There are no more operations to expand the tree, so putting everything together yields
  \[ \Phi(z^{n}w^{k}) = (z+w)^{n} (2zw + w^{2})^{k},  \]
  which proves the statement.
\end{proof}
An immediate consequence of the fact that $\rho\colon \mathcal{L} \to \mathcal{L}$ is
surjective is the following corollary.
\begin{corollary}\label{cor:phi:identity}
  The generating function for plane trees $L(z,w)$ satisfies the functional
  equation
  \[ \Phi(L(z,w)) = L(z,w).  \]
\end{corollary}

We now focus on determining the generating function measuring the change in the tree size
after repeatedly applying the reduction $\rho$.

\begin{proposition}\label{prop:old-leaves:gf}
  Let $r\in\N_{0}$. The bivariate generating function $G_{r}(z,v) =
  G_{r}^{\operatorname{OL}}(z,v)$ enumerating plane trees, where $z$ marks the tree
  size and $v$ marks the size of the $r$-fold cut tree, is given by
  \[ G_{r}(z,v) = \Phi^{r}(L(zv, wv^{2}))|_{w=z^{2}} = L(zB_{r}(z)v, z(B_{r+1}(z) -
    B_{r}(z))v^{2}), \]
  where the $B_{r}(z)$ are the polynomials enumerating binary trees of height $\leq r$
  w.r.t.\ the number of internal nodes.
\end{proposition}
\begin{proof}
  First, note that the size of a tree with $k$ old leaves and $n$ nodes
  that are neither old leaves nor parents thereof is actually
  $n+2k$, as parents of old leaves are not explicitly marked. This explains why we have to
  substitute $w = z^{2}$ in order to arrive at the tree size.

  In contrast to the previous sections, the operator $\Phi$ is already linear and
  multiplicative, meaning that we have
  \[ \Phi^{r}(z^{n}w^{k}) = \Phi^{r}(z)^{n} \Phi^{r}(w)^{k}.  \]
  Investigating the repeated application of $\Phi$ to $z$ and $w$ leads to the recurrences
  \[ \Phi^{r}(z) = \Phi^{r-1}(z) + \Phi^{r-1}(z)^{2} - \Phi^{r-2}(z)^{2} \quad\text{ and
    }\quad \Phi^{r}(w) = \Phi^{r+1}(z) - \Phi^{r}(z)  \]
  for $r\geq 2$ and $r\geq 0$, respectively.
  With the recurrence for the polynomials $B_{r}$ from~\eqref{eq:binary-height-poly} it is easy to prove by induction that
  \[ \Phi^{r}(z)|_{w=z^{2}} = z B_{r}(z)  \]
  for $r \geq 0$. Thus, we also find $\Phi^{r}(w)|_{w=z^{2}} = z(B_{r+1}(z) - B_{r}(z))$.
  Overall, we obtain
  \[ \Phi^{r}(z^{n}w^{k})|_{w=z^{2}} = z^{n+k} B_{r}(z)^{n} (B_{r+1}(z) -
    B_{r}(z))^{k},   \]
  which, by linearity of $\Phi$, proves the proposition.
\end{proof}

For the next step in our analysis, we turn to the random variable $X_{n,r} =
X_{n,r}^{\operatorname{OL}}$ which models the size of the tree 
that results from reducing a random tree $\tau$ with $n$ nodes $r$-times.

As we have $\rho(\innernode) = \innernode$ (and thus no trees vanish completely), the
probability generating function for this random variable is simply
\[ \E v^{X_{n,r}} = \frac{[z^{n}] G_{r}(z,v)}{C_{n-1}}. \]

While the height polynomials $B_{r}(z)$ make it very difficult to obtain general results
for the factorial moments of $X_{n,r}$, special moments like
expectation and variance are no problem, and even a central limit theorem is possible.

\begin{theorem}\label{thm:oldleaves-moments}
  Let $r\in \N_{0}$ be fixed and consider $n\to\infty$. Then the expected tree size after deleting old leaves
  of a tree with $n$ nodes $r$-times and the corresponding variance are given by
  \begin{equation}
    \label{eq:oldleaves-expectation}
    \E X_{n,r} = (2 - B_{r}(1/4)) n - \frac{B_{r}'(1/4)}{8} + O(n^{-1}),
  \end{equation}
  and
  \begin{equation}
    \label{eq:oldleaves-variance}
    \V X_{n,r} = \Bigg(B_{r}(1/4) - B_{r}(1/4)^{2} + \frac{(2 - B_{r}(1/4))
      B_{r}'(1/4)}{2} \Bigg) n + O(1).
  \end{equation}
  All $O$-constants in this theorem depend implicitly on $r$.

  Additionally, the random variable $X_{n,r}$ is asymptotically normally distributed for
  fixed $r\geq 1$, i.e.
  \[ \frac{X_{n,r} - \mu n}{\sqrt{\sigma^{2} n}\,} \convdistr \mathcal{N}(0,1),  \]
  where $\mu = (2 - B_{r}(1/4))$ and $\sigma^{2} = \big(B_{r}(1/4) - B_{r}(1/4)^{2} +
  \frac{(2 - B_{r}(1/4)) B_{r}'(1/4)}{2} \big)$.
\end{theorem}
\begin{proof}
  First of all, we observe that
  Proposition~\ref{proposition:gf-L-explicit} and Proposition~\ref{prop:old-leaves:gf} combined with the
  recursion $B_{r}(z) = 1 + zB_{r-1}(z)^{2}$ allow us to write the bivariate generating
  function as
  \[ G_{r}(z,v) = \frac{1 - \sqrt{1 - 4zv(B_{r}(z)(1-v) + v)}\,}{2}. \]
  The asymptotic expansion for the expected value $\E X_{n,r}$ can now be obtained by
  determining
  \[ \frac{1}{C_{n-1}} [z^{n}] \frac{\partial }{\partial v} G_{r}(z,v)|_{v=1} =
    \frac{1}{C_{n-1}} [z^{n}] \frac{z (2 - B_{r}(z))}{\sqrt{1 -
        4z}\,}. \]

  By means of singularity analysis we find
  \[ \E X_{n,r} = (2 - B_{r}(1/4)) n - \frac{B_{r}'(1/4)}{8} - \Big(\frac{3
      B_{r}'(1/4)}{16} + \frac{3 B_{r}''(1/4)}{128}\Big) n^{-1} + O(n^{-2}), \]
  which proves~\eqref{eq:oldleaves-expectation}. For the second factorial moment we obtain
  \[ \E X_{n,r}^{\underline{2}} = \frac{1}{C_{n-1}} [z^{n}] \frac{\partial^{2} }{\partial
      v^{2}} G_{r}(z,v)|_{v=1} = \frac{1}{C_{n-1}} [z^{n}]
    \Big(\frac{2z^{2}(2-B_{r}(z))}{(1-4z)^{3/2}} +
    \frac{2z(1-B_{r}(z))}{(1-4z)^{1/2}}\Big), \]
  which yields
  \begin{multline*} 
    \E X_{n,r}^{\underline{2}} = (2-B_{r}(1/4))^{2} n^{2} + \Big(2B_{r}(1/4) -
    B_{r}(1/4)^{2} - 2 + \frac{(2 - B_{r}(1/4)) B_{r}'(1/4)}{4}\Big)n \\
    + \frac{(2 - B_{r}(1/4)) B_{r}''(1/4)}{64} - \frac{B_{r}'(1/4)^{2}}{64} -
    \frac{B_{r}(1/4) B_{r}'(1/4)}{8} + O(n^{-1}).
  \end{multline*}
  The variance can now be obtained via $\V X_{n,r} = \E X_{n,r}^{\underline{2}} + \E
  X_{n,r} - (\E X_{n,r})^{2}$, which proves~\eqref{eq:oldleaves-variance}.

  In order to show asymptotic normality of $X_{n,r}$ we investigate the random variable $n
  - X_{n,r}$, which counts the number of nodes that are deleted after reducing some tree $r$
  times. Observe that this quantity can be seen as an additive tree parameter $F_{r}$
  defined recursively by
  \[ F_{r}(\tau_{n}) = F_{r}(\tau_{i_{1}}) + F_{r}(\tau_{i_{2}}) + \cdots +
    F_{r}(\tau_{i_{\ell}}) + f_{r}(\tau_{n}) \quad\text{ and }\quad F_{r}(\innernode) = 0\]
  where $\tau_{n}$ is some tree of size $n$, $\tau_{i_{1}}$ up to $\tau_{i_{\ell}}$ are
  the subtrees rooted at the children of the root of $\tau_{n}$, and $f_{r}\colon
  \mathcal{L} \to \{0,1,\ldots, r-1\}$ is a toll function defined by
  \[ f_{r}(\tau_{n}) = \sum_{j=0}^{r-1} \begin{cases} 1 & \text{ if } \rho^{j}(\tau_{n})
      \text{ has an old leaf attached to its root},\\
      0 & \text{ otherwise, }
    \end{cases} \]
  for $r\geq 1$. Now, as $f_{r}(\tau_{n})$ enumerates the number of old leaves deleted
  from the root of $\tau_{n}$ after $r$ reductions, $F_{r}(\tau_{n})$ equals the total
  number of deleted nodes after $r$ reductions.

  The fact that $r$ is fixed implies that $f_{r}$ is not only bounded, but also a
  so-called \emph{local functional}, meaning that the value of $f_{r}(\tau_{n})$ can
  already be determined from the first $r$ levels of $\tau_{n}$. This is because one
  application of $\rho$ can reduce the distance between the root of the tree and the
  closest old leaf by at most one. Thus all old leaves that are deleted from the root
  during $r$ reductions have to be found within the first $r$ levels of $\tau_{n}$.

  As we have now established that $f_{r}$ is both bounded and a local functional, we are
  able to apply~\cite[Theorem 1.13]{Janson:2016:normality-add-func}, which proves that $n
  - X_{n,r}$ is asymptotically normally distributed. Thus $X_{n,r}$ is
  asymptotically normally distributed as well, which proves the statement.
\end{proof}

\begin{remark}
  In~\cite{Flajolet-Odlyzko:1982}, the asymptotic behavior of a sequence strongly related
  to $B_{r}(1/4)$ was studied: in Section~4, the authors define a sequence $f_{n}$ such that
  $f_{r+1}=\frac12 - \frac{B_{r}(1/4)}{4}$, in our notation. They prove the
  asymptotic expansion $f_{n} = \frac{1}{n + \log n + O(1)}$. This allows us to conclude
  that the asymptotic behavior of $B_{r}(1/4)$ can be described as
  \[ B_{r}(1/4) = 2 - \frac{4}{r} + \frac{4 \log r}{r^{2}} + O(r^{-2}) \]
  for $r \to \infty$.
\end{remark}

\section{Cutting Old Paths}
\label{sec:cutting-old-paths}

\subsection{The Expansion
  Operator}\label{sec:old-paths:expansion-operator}
As in previous sections, we adapt the ``old leaves'' reduction
to remove all ``old paths''. That is, the tree reduction $\rho\colon \mathcal{L} \to
\mathcal{L}$ in this section reduces a tree by removing all paths that
end in an old leaf. This operation is illustrated in
Figure~\ref{fig:old-paths:illustration}, where $\blacksquare$ marks old leaves and
$\innernode$ marks all nodes that are neither old leaves nor parents thereof.

\begin{figure}[ht]
  \centering
          \[
    \begin{tikzpicture}[thick, scale=0.75, baseline={([yshift=-0.5em]current bounding
      box.center)}]
      \node[draw, circle] {}
      child[gray, dashed] {node[draw, circle] {} child {node[draw, circle] {} child{node[draw,
            rectangle] {}}}}
      child {node[draw, circle] {}
        child[gray, dashed] {node[draw, rectangle] {}}
        child[] {node[draw, rectangle] {}}
      }
      child[] {node[draw, circle] {} child[gray, dashed] {node[draw, rectangle] {}}};
    \end{tikzpicture}
    \quad \mapsto \quad
    \begin{tikzpicture}[thick, scale=0.75, baseline={([yshift=-0.5em]current bounding
      box.center)}]
      \node[draw, circle] {}
      child[gray, dashed] {node[draw, circle] {}
        child[] {node[draw, rectangle] {}}
      }
      child[gray, dashed] {node[draw, rectangle] {} };
    \end{tikzpicture}
    \quad \mapsto \quad
    \begin{tikzpicture}[thick, scale=0.75, baseline={([yshift=-0.5em]current bounding
      box.center)}]
      \node[draw, circle] {
      };
    \end{tikzpicture}
    \]
  \caption{Illustration of the ``cutting old paths''-operator $\rho$}
  \label{fig:old-paths:illustration}
\end{figure}
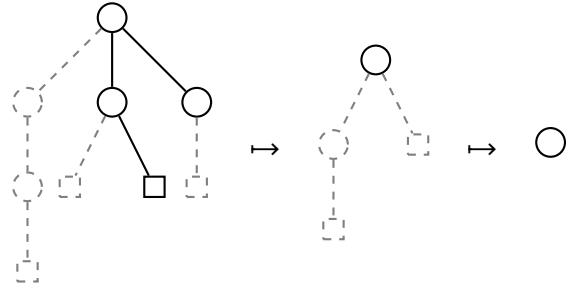

Obviously, we also need the combinatorial class of paths $\mathcal{P}$ for our
analysis. The bivariate generating function of $\mathcal{P}$ is given by $P = P(z,w) =
\frac{w}{1-z}$, where
$w$ and $z$ mark $\blacksquare$ and $\innernode$,
respectively. Also, we omit the arguments of $P$ for the sake of
readability. Now, we determine the shape of the expansion operator $\Phi$.

\begin{proposition}\label{prop:old-paths:expansion}
  Let $\mathcal{F}\subseteq\mathcal{L}$ be a family of plane trees with bivariate
  generating function $f(z,w)$, where $z$ marks nodes that are neither old leaves
  nor parents thereof and $w$ marks old leaves. Then the generating function
  for $\rho^{-1}(\mathcal{F})$, the family of trees whose reduction is in $\mathcal{F}$,
  is given by
  \begin{equation}
    \label{eq:old-paths:expansion}
    \Phi(f(z,w)) = f(z+P, zP + P^{2}).
  \end{equation}
\end{proposition}
\begin{proof}
  With linearity of the operator $\Phi$ being obvious from a combinatorial point of view,
  we only have to investigate the expansion of any tree represented by $z^{n}w^{k}$,
  i.e.\ a tree with $n$ nodes that are neither old leaves nor parents thereof and $k$ old
  leaves.

  There are two options to expand an old leaf $\blacksquare$:
  \begin{itemize}
  \item[--] either appending an old path to the parent of $\blacksquare$, which
    turns the old leaf into $\innernode$,
  \item[--] or an old path is appended to both the parent of $\blacksquare$ and to
    $\blacksquare$ itself.
  \end{itemize}
  Note that just appending an old path to $\blacksquare$ is not a valid expansion as this
  introduces ambiguity. This is the same argument that we also used in the proof of
  Proposition~\ref{prop:paths:expansion}. Overall, this means that $\Phi$ has to map $w$
  to $zP + P^{2}$.

  On the other hand, the nodes represented by $\innernode$ can optionally be expanded by
  attaching an old path. Otherwise they stay as they are. Overall, this implies $\Phi(z) =
  z+P$.

  Putting everything together, we immediately arrive at the statement of the Proposition.
\end{proof}
Analogously to the previous reductions, surjectivity of $\rho\colon\mathcal
L\to\mathcal L$ implies the following
corollary.
\begin{corollary}
  The generating function for plane trees $L(z,w)$ satisfies the functional
  equation
  \[ \Phi(L(z,w)) = L(z,w).  \]
\end{corollary}

In order to carry out a detailed analysis of this reduction, we need information about the
iterated application of $\Phi$ to $L(zv_{I}, wv_{L}^{2})$, which leads to the generating function
$G_{r}(z,v_{I},v_{L}^2)$ measuring the change in the tree size after $r$ applications of the
reduction. The following proposition deals with determining this generating function.

\begin{proposition}\label{prop:old-paths:gf}
  Let $r\in \mathbb{N}_{0}$. The trivariate generating function $G_{r}(z,v_{I},v_{L}^{2}) =
  G_{r}^{\operatorname{OP}}(z,v_{I},v_{L}^{2})$ enumerating plane trees, where $z$ marks the tree
  size, $v_{L}$ marks all old leaves, and $v_{I}$ marks all nodes that are neither old
  leaves nor parents thereof, is given by
  \[ G_{r}(z,v_{I},v_{L}^{2}) = \Phi^{r}(L(zv_{I}, wv_{L}^{2}))|_{w=z^{2}} =
    L\bigg(\frac{u(1-u^{r+1})}{(1+u)(1-u^{r+2})} v_{I}, \frac{u^{r+2} (1-u)^{2}}{(1+u)^{2}
      (1-u^{r+2})^{2}} v_{L}^{2}\bigg),  \]
  where $z = u/(1+u)^{2}$.
\end{proposition}
\begin{proof}
  Observe that the operator $\Phi$ is already linear and multiplicative, which is why we
  can concentrate on finding suitable expressions for $\Phi^{r}(z)$ and $\Phi^{r}(w)$.

  First of all, for $r\geq 1$ the recurrences
  \[ \Phi^{r}(z) = \Phi^{r-1}(z) + \Phi^{r-1}(P), \qquad \Phi^{r}(w) =
    \Phi^{r-1}(P)\Phi^{r}(z)  \]
  follow immediately from~\eqref{eq:old-paths:expansion}. Furthermore, the relation
  \[ \Phi^{r}(P) = P \prod_{j=1}^{r} \frac{\Phi^{j}(z)}{1 - \Phi^{j}(z)} \]
  can easily be proved by induction. Then, by setting $f_{r} \coloneqq \Phi^{r}(z)|_{w=z^{2}}$
  the recurrences above translate to
  \[ f_{r} = f_{r-1} + z \prod_{j=0}^{r-1} \frac{f_{j}}{1 - f_{j}}.  \]
  As a next step, we show by induction that $f_{r}$ can be expressed in terms of Fibonacci
  polynomials as
  \[ f_{r} = \frac{z F_{r+1}(-z)}{F_{r+2}(-z)},  \]
  where in particular~\eqref{eq:fibonacci} was used. As a consequence, we find
  \[ \Phi^{r}(P)|_{w = z^{2}} = f_{r+1} - f_{r} = \frac{z F_{r+2}(-z)}{F_{r+3}(-z)} -
    \frac{z F_{r+1}(-z)}{F_{r+2}(-z)} = \frac{z^{r+2}}{F_{r+2}(-z) F_{r+3}(-z)}. \]
  This allows us to express $g_{r} \coloneqq \Phi^{r}(w)|_{w=z^{2}}$ as
  \[ g_{r} = \Phi^{r-1}(P)|_{w=z^{2}} \cdot  f_{r} = \frac{z^{r+2}}{F_{r+2}(-z)^{2}}.  \]
  Finally, as we have $\Phi^{r}(z^{n} w^{k})|_{w=z^{2}} = f_{r}^{n} g_{r}^{k}$,
  substituting $z = u/(1+u)^{2}$ and using~\eqref{eq:fib-poly-u} completes the proof.
\end{proof}

\subsection{Analysis of Tree Size and Related Parameters}

We investigate the behavior of the random variable $X_{n,r} = X_{n,r}^{\operatorname{OP}}$
which models the number of nodes remaining after reducing a random tree
$\tau$ with $n$ nodes $r$-times. The tree $\tau$ is chosen uniformly
among all trees of size $n$. Analogously to the ``old leaf''-reduction from the previous
section, we also have $\rho(\innernode) = \innernode$ for the ``old path''-reduction,
meaning that no trees vanish completely. For the sake of convenience we set $G_{r}(z,v) \coloneqq
G_{r}(z,v,v^{2})$, allowing us to write the probability generating function of
$X_{n.r}$ as
\begin{equation*}
  \E v^{X_{n,r}}=\frac{[z^{n}]G_{r}(z,v)}{C_{n-1}}.
\end{equation*}
With the help of Proposition~\ref{prop:old-paths:gf}, it is easy to obtain expressions for
the factorial moments $\E X_{n,r}^{\underline{d}}$ for fixed $d$ by differentiating $G_{r}(z,v)$
$d$-times with respect to $v$ and setting $v=1$ afterwards. General expressions
for $d\ge 2$ (coinciding with the value given for $d=2$) are available but less pleasant.

\begin{lemma}\label{lemma:moments-cut-old-paths}
  The factorial moments of $X_{n,r}$ are
  \begin{align*}
    \E X_{n,r}&=\frac{1}{C_{n-1}}[z^{n}] \frac{u (1 + u^{r+1})}{(1+u) (1-u^{r+2})},\\
    \E X_{n,r}(X_{n,r}-1)&=\frac{2}{C_{n-1}}[z^{n}] \frac{(1+u) u^{r+2}}{(1-u) (1-u^{r+2})^{2}}
  \end{align*}
  and
  \begin{align*}
    \E X_{n, r}^{\underline d} &=
    \frac{d!}{C_{n-1}}[z^n]\frac{1-u}{1+u}\biggl(\frac{u(1+u^{r+1})}{(1-u)(1-u^{r+2})}
    + \frac{u}{1-u}\sqrt{\frac{1-u^r}{1-u^{r+2}}\,}\biggr)^d\\
    &\qquad\times
    \tilde N_{d-1}\biggl(\frac{2u^{2r+2} - u^{r+2}+2u^{r+1}-u^r+2}{(1+u)^2u^{r}}
     + \frac{2(1+u^{r+1})(1-u^{r+2})}{u^r(1+u)^2}\sqrt{\frac{1-u^r}{1-u^{r+2}}\,}\biggr)
  \end{align*}
  for $d\ge 2$.
\end{lemma}

\begin{proof}
  The expressions for $d\in\{1, 2\}$ can be obtained by differentiation. We
  consider the general case here.

  We use the abbreviations
  \begin{equation*}
    a=\frac{u(1-u^{r+1})}{(1+u)(1-u^{r+2})},\qquad
    b=\frac{u^{r+2}(1-u)^2}{(1+u)^2(1-u^{r+2})^2}, \qquad
    \Delta = \frac{1-u}{1+u}.
  \end{equation*}
  By the same argument as in the proof of Proposition~\ref{prop:cut-leaves:moments}, we
  have
  \begin{equation*}
    \frac{\partial^d}{\partial v^d} G_r(z, v)\Bigr|_{v=1} = d![q^d]L(a(1+q), b(1+q)^2).
  \end{equation*}
  By using \eqref{eq:gf-old-leaves}, we rewrite $L(a(1+q), b(1+q)^2)$ as
  \begin{align*}
    L(a(1+q), b(1+q)^2) &= \frac{1-\sqrt{(1-2a-2aq)^2-4b(1+q)^2}\,}{2}\\
    &=\frac{1-\sqrt{(1-2a)^2-4b -2q(2a(1-2a)+4b)+q^2(4a^2-4b)}\,}2.
  \end{align*}

  We have
  \begin{equation*}
    (1-2a)^2-4b = \Delta^2, \qquad \sqrt{\frac{4a^2-4b}{\Delta^2}}\, = \frac{2u}{1-u}\sqrt{\frac{1-u^r}{1-u^{r+2}}}\,.
  \end{equation*}
  We choose $\alpha$ and $\beta$ such that
  \begin{equation*}
    \alpha+\beta=\frac{2a(1-2a)+4b}{\Delta^2}, \qquad
    \alpha-\beta= \sqrt{\frac{4a^2-4b}{\Delta^2}}\,.
  \end{equation*}
  This results in
  \begin{align*}
    L(a(1+q), b(1+q)^2) &= \Delta\frac{\frac1\Delta-1+q(\alpha-\beta)}{2} + \Delta\frac{1-q(\alpha-\beta)-\sqrt{1-2q(\alpha+\beta)+q^2(\alpha-\beta)}\,}{2}\\
    &=\Delta\frac{\frac1\Delta-1+q(\alpha-\beta)}{2} + \Delta T(\alpha q, \beta q).
  \end{align*}
  Using~\eqref{eq:gf-narayana:quotient} to extract the coefficient of $q^d$
  for $d\ge 1$ yields
  \begin{equation*}
    \frac{\partial^d}{\partial v^d} G_r(z, v)\Bigr|_{v=1} =
    d!\Delta\Bigl(\frac{\alpha-\beta}2\iverson{d=1} + \alpha^d N_{d-1}\Bigl(\frac{\beta}{\alpha}\Bigr)\Bigr).
  \end{equation*}
  Inserting everything concludes the proof of the proposition.
\end{proof}

\begin{corollary}
  The expected value of $X_{n+1,r}$ is explicitly given by
  \begin{equation*}
    \E X_{n+1,r}=\frac{1}{C_{n}} \bigg(\binom{2n}{n} + \sum_{j\geq 0} \biggl(\binom{2n}{n -
      (j+1)(r+2) + 1} - \binom{2n}{n - j(r+2) - 1}\biggr) \bigg)
  \end{equation*}
\end{corollary}
\begin{proof}
From Lemma~\ref{lemma:moments-cut-old-paths}, we obtain
\begin{align*}
  C_{n}\E
  X_{n+1,r}=[z^{n+1}]\frac{(1+u^{r+1})u}{(1+u)(1-u^{r+2})}
\end{align*}
and proceeding as in Corollary~\ref{corollary:leaves-exact-exp-val} we
obtain the given result.
\end{proof}

By expanding the expressions in Lemma~\ref{lemma:moments-cut-old-paths}
and using singularity analysis, we obtain the asymptotic growth of the expected value and
the variance.
\begin{theorem}\label{thm:moments-asy-old-paths}
  Let $r\in\N$ be fixed and consider $n\to\infty$. Then the expected size and the corresponding variance of an
  $r$-fold cut plane tree are given by
  \begin{equation*}
    \E X_{n,r} = \frac{2n}{r+2} - \frac{r(r+1)}{3 (r+2)} + O(n^{-1}),
  \end{equation*}
  and
  \begin{equation*}
    \V X_{n,r} = \frac{2r(r+1)}{3(r+2)^{2}}n + O(1).
  \end{equation*}
  For $d\ge 3$, the $d$th factorial moment is
  \begin{equation*}
    \E X_{n,r}^{\underline{d}}=\frac{2^{d - 1} d  }{(2  d - 3)(r + 2)^d} n^{d} + \binom{2d-5}{d-2}\frac{\sqrt{r\pi}\,  d  }{2^{d-3}(r+2)^{d-1/2}} n^{d - \frac{1}{2}} + O(n^{d - 1}).
  \end{equation*}
   All $O$-constants in this theorem depend implicitly on $r$.
\end{theorem}

Besides the analysis of the tree size, we are also interested in how the
numbers of nodes represented by $\blacksquare$ and by $\innernode$ develop when the tree is
reduced repeatedly. Formally, this means that we consider the random
variables $X_{n,r}^{\blacksquare}$ and $X_{n,r}^{\filledcirc}$ counting the number of old leaves and
the number of all nodes that are neither old leaves nor parents thereof, respectively. By
construction, the relation
\begin{equation}\label{eq:oldpaths:split-leaves-inner}
  X_{n,r} = 2\cdot X_{n,r}^{\blacksquare} + X_{n,r}^{\filledcirc}
\end{equation}
holds.

The bivariate generating functions corresponding to these random variables can be obtained
directly from Proposition~\ref{prop:old-paths:gf}. We have
\[ G_{r}^{\blacksquare}(z,v) = G_{r}(z,1,v),\qquad G_{r}^{\filledcirc}(z,v) = G_{r}(z,v,1). \]

In contrast to $X_{n,r}$, the $d$th factorial moments for
$X_{n,r}^{\blacksquare}$ and $X_{n,r}^{\filledcirc}$ have simpler expressions.

\begin{proposition}\label{prop:old-paths:factorial}
  Let $d\in\N$. Then the $d$th factorial moments of $X_{n,r}^{\blacksquare}$ and
  $X_{n,r}^{\filledcirc}$ are given by
  \begin{equation}\label{eq:old-paths:factorial-leaves}
    \E {X_{n,r}^{\blacksquare}}^{\underline{d}} = \frac{(2d-2)^{\underline{d-1}}}{C_{n-1}}
    [z^{n}]\frac{1-u}{1+u} \frac{u^{rd + 2d}}{(1-u^{r+2})^{2d}}
  \end{equation}
  and
  \begin{equation}\label{eq:old-paths:factorial-rest:d1}
    \E {X_{n,r}^{\filledcirc}} = \frac{1}{C_{n-1}} [z^{n}] \frac{u (1 - u^{r+1}) (1 +
      u^{r+2})}{(1+u) (1 - u^{r+2})^{2}}
  \end{equation}
  as well as
  \begin{equation}
    \label{eq:old-paths:factorial-rest}
    \E {X_{n,r}^{\filledcirc}}^{\underline{d}} = \frac{1}{C_{n-1}} [z^{n}] \frac{(1-u^{r+1})^{d} u^{d}
      2^{d} d!}{(1-u)^{d-1}(1+u)(1-u^{r+2})^{2d}} \tilde N_{d-1}(u^{r+2})
  \end{equation}
  for $d > 1$.
\end{proposition}

\begin{proof}[Proof of Proposition~\ref{prop:old-paths:factorial}]
  As in the proof of Lemma~\ref{lemma:moments-cut-old-paths}, we use the abbreviations
  \begin{equation*}
    a=\frac{u(1-u^{r+1})}{(1+u)(1-u^{r+2})},\qquad
    b=\frac{u^{r+2}(1-u)^2}{(1+u)^2(1-u^{r+2})^2}, \qquad
    \Delta = \frac{1-u}{1+u}.
  \end{equation*}

  Then, using \eqref{eq:L-sqrt-z}, we get
  \begin{equation*}
    \frac{\partial^d}{\partial v^d}G_r^\blacksquare(z, v)=
    \frac{\partial^d}{\partial v^d} L(a, bv)
    =-\frac{1-2a}{2}\Bigl(\frac12\Bigr)^{\underline{d}}\Bigl(1-\frac{4bv}{(1-2a)^2}\Bigr)^{1/2-d}\frac{(-4b)^d}{(1-2a)^{2d}}.
  \end{equation*}
  Setting $v=1$ and using the fact that
  \begin{equation}\label{eq:miracle-old-paths}
    (1-2a)^2-4b=\Delta^2
  \end{equation}
  proves \eqref{eq:old-paths:factorial-leaves}.

  For deriving $\partial^d/(\partial v)^d G_r^{\filledcirc}(z, v)$, we proceed
  as in the proof of Proposition~\ref{prop:cut-leaves:moments}. The crucial
  identity is
  \begin{equation*}
    L(a(1+q), b) = \Delta\frac{\frac1\Delta-1+(\alpha'-\beta')q}2 + \Delta T(\alpha'q, \beta'q)
  \end{equation*}
  with
  \begin{equation*}
    \alpha'=\frac{2u(1-u^{r+1})}{(1-u^{r+2})^2(1-u)}, \qquad
    \frac{\beta'}{\alpha'} = u^{r+2}.
  \end{equation*}
  This implies \eqref{eq:old-paths:factorial-rest:d1} and \eqref{eq:old-paths:factorial-rest}.

\end{proof}

As in Section~\ref{sec:cut-leaves:analysis}, the above proof exhibits some
identities:

\begin{remark}\label{lem:old-leaves:series-identities}
  For $d\in \Z_{\geq 1}$, the power series identities
\begin{equation}\label{eq:old-leaves:series-identities:leaves}
  \sum_{\substack{n\geq 0 \\ k\geq 1}} \frac{u^{n+k} x^{k} (1-x)^{n}
    (1-u)^{2k}}{(1+u)^{n+2k} (1-ux)^{n+2k}} k^{\underline{d}} C_{k-1} \binom{n+2k-2}{n} 2^{n} =
  (2d-2)^{\underline{d-1}} \frac{u^{d} x^{d} (1-u)}{(1-ux)^{2d} (1+u)}
\end{equation}
and
\begin{equation}\label{eq:old-leaves:series-identities:rest}
  \sum_{\substack{n\geq 0\\ k\geq 1}} \frac{u^{n+k} x^{k} (1-x)^{n} (1-u)^{2k}}{(1+u)^{n+2k} (1-ux)^{n+2k}}
  n^{\underline{d}} C_{k-1} \binom{n+2k-2}{n} 2^{n} = \frac{(1-x)^{d} u^{d} 2^{d}
    d!\, \tilde N_{d-1}(ux)}{(1-u)^{d-1} (1+u) (1-ux)^{2d}}
\end{equation}
hold.
\end{remark}
\begin{proof}
  We replace $u^{r+1}$ by $x$ in the proof of
  Proposition~\ref{prop:old-paths:factorial} and expand $L$ by
  \eqref{eq:L-full-expansion}.
\end{proof}

The asymptotic behavior for the factorial moments of $X_{n,r}^{\blacksquare}$ and
$X_{n,r}^{\filledcirc}$ can now be extracted quite straightforward by means of singularity
analysis from the representation given in Proposition~\ref{prop:old-paths:factorial}.

\begin{theorem}\label{thm:old-paths:factorial-asy}
  Let $r\in \N_{0}$ be fixed and consider $n\to\infty$. Then the expected number of old
  leaves as well as the expected number of nodes that are neither old leaves nor parents
  thereof in an $r$-fold ``old path''-reduced tree and the corresponding variances are
  given by the asymptotic expansions
  \begin{align}
    \label{eq:old-paths:factorial-asy:exp}
    \E X_{n,r}^{\blacksquare} &= \frac{1}{(r+2)^{2}}n + \frac{(r+3)(r+1)}{6(r+2)^{2}} + O(n^{-1}),\\ \E X_{n,r}^{\filledcirc}& =
    \frac{2(r+1)}{(r+2)^{2}} n - \frac{(r^{2} + 3r + 3)(r+1)}{3(r+2)^{2}} + O(n^{-1}),\nonumber\\
    \label{eq:old-paths:factorial-asy:var}
    \V X_{n,r}^{\blacksquare} &= \frac{(r+3)(r+1)}{3(r+2)^{4}}n +
    O(1),\\ 
\V X_{n,r}^{\filledcirc} &= \frac{2(r^{3} +
      4r^{2} + 6r + 6)(r+1)}{3(r+2)^{4}} n + O(1).\nonumber
  \end{align}
  Additionally, for fixed $d\geq 2$ the behavior of the factorial moments of
  $X_{n,r}^{\blacksquare}$ and $X_{n,r}^{\filledcirc}$ is given by
  \begin{equation}
    \label{eq:old-paths:factorial-asy:leaves-full}
    \E {X_{n,r}^{\blacksquare}}^{\underline{d}} = \frac{1}{(r+2)^{2d}} + O(n^{d-1})
  \end{equation}
  and
  \begin{equation}
    \label{eq:old-paths:factorial-asy:rest-full}
    \E {X_{n,r}^{\filledcirc}}^{\underline{d}} = \frac{2^{d} (r+1)^{d}}{(r+2)^{2d}} n^{d}
    + O(n^{d-1}),
  \end{equation}
  respectively.
   All $O$-constants in this theorem depend implicitly on $r$.
\end{theorem}

\subsection{Total number of old paths}
\label{sec:old-paths:total}

Similarly to our approach for counting the total number of paths required to construct a
given tree from Section~\ref{sec:paths:totalnumber}, we can also analyze the number of
``old path''-segments within a random tree of size $n$. Formally, this corresponds to an
analysis of the random variable $S_{n} \coloneqq \sum_{r\geq 0} X_{n,r}^{\blacksquare}$.

\begin{theorem}\label{thm:old-paths:total}
  The expected number of ``old path'' segments within a uniformly random tree of size $n$
  is given asymptotically by
  \begin{equation}
    \label{eq:old-paths:total:asy}
    \E S_{n} = \Big(\frac{\pi^{2}}{6} - 1\Big) n - \frac{\pi^{2}}{36} - \frac{1}{12} -
    \frac{\pi^{2}}{120n} + O(n^{-2})
  \end{equation}
  for $n\to\infty$.
\end{theorem}
\begin{proof}
  As we have $S_{n} = \sum_{r\geq 0} X_{n,r}^{\blacksquare}$, we can
  use~\eqref{eq:old-paths:factorial-leaves} to write
  \[ \E S_{n} = \sum_{r\geq 0} \E X_{n,r}^{\blacksquare} = \frac{1}{C_{n-1}} [z^{n}]
    \frac{1-u}{1+u} \sum_{r\geq 0} \frac{u^{r+2}}{(1-u^{r+2})^{2}}. \]
  The main part of this analysis consists of determining an appropriate expansion of the
  sum in the last equation via the Mellin transform.

  By setting $u = e^{-t}$ and by means of expanding via the geometric series, we find
  \[ \sum_{r\geq 0} \frac{u^{r+2}}{(1-u^{r+2})^{2}} = \sum_{r,\lambda \geq 0} \lambda
    u^{\lambda(r+2)} = \sum_{r,\lambda\geq 0} \lambda e^{-t \lambda (r+2)} \eqqcolon f(t).  \]
  It is easy to determine the corresponding Mellin transform
  \[ f^{*}(s) = \Gamma(s) \zeta(s-1) (\zeta(s) - 1) \]
  with fundamental strip $\langle 2, \infty\rangle$. The poles of
  $f^{*}(s)$ are located at $s\in \{2,1\}\cup -2\N_{0}$. As this function behaves very nicely
  along vertical lines because of the exponential decay and the polynomial growth of the
  gamma function and the zeta function, respectively, we can use the inversion theorem
  to find
  \[ f(t) = \frac{1}{2\pi i} \int_{3-i\infty}^{3+i\infty} f^{*}(s) t^{-s}~ds \]
  for $t \to 0$. Analyticity in a larger (complex) region can be obtained analogously to the
  approach in the proof of Theorem~\ref{thm:paths:total}.

  Shifting the line of integration to $\Re(s) = -5$ and collecting residues, we find
  \[ f(t) = \sum_{p\in \{2,1,0,-2,-4\}} \Res(f^{*}(s), s=p)  t^{-p} +
    \frac{1}{2\pi i}\int_{-5-i\infty}^{-5+i\infty} f^{*}(s)t^{-s}~ds.  \]
  As in the proof of Theorem~\ref{thm:paths:total}, the integral can be
  estimated with an error of $O(|t|^{5})$. However, for the sake of simplicity, we will use
  the contribution from the singularity at $s = -4$ as the expansion error. Effectively,
  we obtain
  \[ f(t) = \Big(\frac{\pi^{2}}{6} - 1\Big) t^{-2} - \frac{1}{2} t^{-1} + \frac{1}{8} -
    \frac{1}{240} t^{2} + O(t^{4}) \]
  for $t\to 0$. Multiplication with the factor $\frac{1-u}{1+u}$, expansion of everything
  in terms of $z\to 1/4$, carrying out singularity analysis, and normalizing the result by
  dividing by $C_{n-1}$ yields the result.
\end{proof}

\section{Future Work}
It seems likely that similar results also hold for reductions where
one can cut a different structure as long as it is allowed to cut a
single leaf. An example is cutting either single leaves or cherries (a
root with two children). At least a formulation as an operator as in~\eqref{eq:cut-leaves-op} seems
possible in general. How much information about the moments and the central limit
theorem can be extracted from that may vary (as it varies in this
article already). Also the case of cutting old structures might be
more difficult to handle in general.

\bibliographystyle{amsplainurl}
\bibliography{bib/cheub}

\providecommand{\Submitted}{Submitted} \providecommand{\availableat}{ available
  at } \providecommand{\alsoavailableat}{ also available at }
  \providecommand{\evavailableat}{earlier version available at }
  \providecommand{\toappearin}{To appear in } \providecommand{\toappear}{to
  appear} \providecommand{\inpreparation}{in preparation}
  \providecommand{\doi}[1]{\href{http://dx.doi.org/#1}{\path{doi:#1}}}
  \providecommand{\lowercaseforams}{}
  \providecommand{\etc}{\emph{etc.}}\def\cprime{$'$}
\providecommand{\bysame}{\leavevmode\hbox to3em{\hrulefill}\thinspace}
\providecommand{\MR}{\relax\ifhmode\unskip\space\fi MR }
\providecommand{\MRhref}[2]{%
  \href{http://www.ams.org/mathscinet-getitem?mr=#1}{#2}
}
\providecommand{\href}[2]{#2}
\begin{thebibliography}{10}

\bibitem{Callan:2012:kreweras-narayana-identity}
David Callan, \href{https://arxiv.org/abs/1203.3999}{\emph{Kreweras's
  {N}arayana number identity has a simple {D}yck path interpretation}},
  arXiv:1203.3999 [math.CO], 2012.

\bibitem{Chen-Deutsch-Elizalde:2006:old-leaves}
William Y.~C. Chen, Emeric Deutsch, and Sergi Elizalde,
  \href{http://dx.doi.org/10.1016/j.ejc.2004.12.008}{\emph{Old and young leaves
  on plane trees}}, European J. Combin. \textbf{27} (2006), no.~3, 414--427.
  \MR{2206476}

\bibitem{Bruijn-Knuth-Rice:1972}
Nicolaas~G. de~Bruijn, Donald~E. Knuth, and Stephen~O. Rice, \emph{The average
  height of planted plane trees}, Graph theory and computing, Academic Press,
  New York, 1972, pp.~15--22. \MR{0505710 (58 \#21737)}

\bibitem{Vauchaussande:1985:thesis}
Mireille~Vauchaussade de~Chaumont, \emph{Nombre de {S}trahler des arbres,
  languages alg\'ebrique et d\'enombrement de structures secondaires en
  biologie mol\'eculaire}, Doctoral thesis, Universit\'{e} de Bordeaux I, 1985.

\bibitem{SageMath:2016:7.4}
The~SageMath Developers, \emph{{SageMath} {M}athematics {S}oftware ({V}ersion
  7.4)}, 2016, \url{http://www.sagemath.org}.

\bibitem{NIST:DLMF:v1.0.13}
\href{http://dlmf.nist.gov/}{\emph{{NIST} {D}igital library of mathematical
  functions}}, \url{http://dlmf.nist.gov/}, Release 1.0.13 of 2016-09-16, 2016,
  Frank W.~J. Olver, Adri B. {Olde Daalhuis}, Daniel W. Lozier, Barry I.
  Schneider, Ronald F. Boisvert, Charles W. Clark, Bruce R. Miller and Bonita
  V. Saunders, eds.

\bibitem{Drmota:2009:random}
Michael Drmota, \href{http://dx.doi.org/10.1007/978-3-211-75357-6}{\emph{Random
  trees}}, SpringerWienNewYork, 2009. \MR{2484382 (2010i:05003)}

\bibitem{Drmota:2015:trees-handbook}
\bysame, \emph{Trees}, Handbook of enumerative combinatorics, Discrete Math.
  Appl. (Boca Raton), CRC Press, Boca Raton, FL, 2015, pp.~281--334.
  \MR{3409345}

\bibitem{Flajolet-Odlyzko:1982}
Philippe Flajolet and Andrew Odlyzko,
  \href{http://dx.doi.org/10.1016/0022-0000(82)90004-6}{\emph{The average
  height of binary trees and other simple trees}}, J. Comput. System Sci.
  \textbf{25} (1982), no.~2, 171--213. \MR{680517 (84a:68056)}

\bibitem{Flajolet-Odlyzko:1990:singul}
\bysame, \href{http://dx.doi.org/10.1137/0403019}{\emph{Singularity analysis of
  generating functions}}, SIAM J. Discrete Math. \textbf{3} (1990), 216--240.
  \MR{MR1039294 (90m:05012)}

\bibitem{Flajolet-Raoult-Vuillemin:1979:register}
Philippe Flajolet, Jean-Claude Raoult, and Jean Vuillemin,
  \href{http://dx.doi.org/10.1016/0304-3975(79)90009-4}{\emph{The number of
  registers required for evaluating arithmetic expressions}}, Theoret. Comput.
  Sci. \textbf{9} (1979), no.~1, 99--125. \MR{535127}

\bibitem{Flajolet-Sedgewick:ta:analy}
Philippe Flajolet and Robert Sedgewick,
  \href{http://dx.doi.org/10.1017/CBO9780511801655}{\emph{Analytic
  combinatorics}}, Cambridge University Press, Cambridge, 2009.

\bibitem{Hackl-Heuberger-Prodinger:2016:reduc-binar}
Benjamin Hackl, Clemens Heuberger, and Helmut Prodinger,
  \href{https://arxiv.org/abs/1612.07286}{\emph{Reductions of binary trees and
  lattice paths induced by the register function}}, arXiv:1612.07286 [math.CO],
  2016.

\bibitem{Hackl-Kropf-Prodinger:2017:iterat-cuttin}
Benjamin Hackl, Sara Kropf, and Helmut Prodinger,
  \href{http://dx.doi.org/10.1137/1.9781611974775.6}{\emph{Iterative cutting
  and pruning of planar trees}}, Proceedings of the Fourteenth Workshop on
  Analytic Algorithmics and Combinatorics (ANALCO) (Philadelphia PA), SIAM,
  2017, pp.~66--72.

\bibitem{Heuberger-Kropf:2016:higher-dimen}
Clemens Heuberger and Sara Kropf,
  \href{https://arxiv.org/abs/1609.09599}{\emph{Higher dimensional quasi-power
  theorem and {B}erry--{E}sseen inequality}}, arXiv:1609.09599 [math.PR], 2016.

\bibitem{Hwang:1998}
Hsien-Kuei Hwang, \href{http://dx.doi.org/10.1006/eujc.1997.0179}{\emph{On
  convergence rates in the central limit theorems for combinatorial
  structures}}, European J. Combin. \textbf{19} (1998), 329--343.

\bibitem{Janson:2006:random-cutting}
Svante Janson, \href{http://dx.doi.org/10.1002/rsa.20086}{\emph{Random cutting
  and records in deterministic and random trees}}, Random Structures Algorithms
  \textbf{29} (2006), no.~2, 139--179. \MR{2245498}

\bibitem{Janson:2016:normality-add-func}
\bysame, \href{http://dx.doi.org/10.1002/rsa.20568}{\emph{Asymptotic normality
  of fringe subtrees and additive functionals in conditioned {G}alton-{W}atson
  trees}}, Random Structures Algorithms \textbf{48} (2016), no.~1, 57--101.
  \MR{3432572}

\bibitem{Kemp:1980:note-stack-size}
Rainer Kemp, \href{http://dx.doi.org/10.1007/BF01933188}{\emph{A note on the
  stack size of regularly distributed binary trees}}, BIT \textbf{20} (1980),
  no.~2, 157--162. \MR{583031}

\bibitem{Kirschenhofer-Prodinger:1988:digital-search-trees}
Peter Kirschenhofer and Helmut Prodinger,
  \href{http://dx.doi.org/10.1016/0304-3975(88)90023-0}{\emph{Further results
  on digital search trees}}, Theoret. Comput. Sci. \textbf{58} (1988), no.~1-3,
  143--154, Thirteenth International Colloquium on Automata, Languages and
  Programming (Rennes, 1986). \MR{963259}

\bibitem{Meir-Moon:1970:cutting-random-trees}
Amram Meir and John~W. Moon, \emph{Cutting down random trees}, J. Austral.
  Math. Soc. \textbf{11} (1970), 313--324. \MR{0284370}

\bibitem{Panholzer:2006:cutting-trees}
Alois Panholzer,
  \href{http://dx.doi.org/10.2989/16073600609486160}{\emph{Cutting down very
  simple trees}}, Quaest. Math. \textbf{29} (2006), no.~2, 211--227.
  \MR{2233368}

\bibitem{Prodinger:1983:height-planted}
Helmut Prodinger,
  \href{http://math.sun.ac.za/~prodinger/pdffiles/height_revisited.pdf}{\emph{The
  height of planted plane trees revisited}}, Ars Combin. \textbf{16} (1983),
  no.~B, 51--55. \MR{737109 (85i:05085)}

\bibitem{Viennot:2002:strah-dyck}
Xavier~G\'{e}rard Viennot,
  \href{http://dx.doi.org/http://dx.doi.org/10.1016/S0012-365X(01)00265-5}{\emph{A
  {S}trahler bijection between {D}yck paths and planar trees}}, Discrete Math.
  \textbf{246} (2002), no.~1--3, 317--329, Formal Power Series and Algebraic
  Combinatorics 1999.

\bibitem{Wagner:2015:centr-limit}
Stephan Wagner,
  \href{http://dx.doi.org/10.1017/S0963548314000443}{\emph{Central limit
  theorems for additive tree parameters with small toll functions}}, Combin.
  Probab. Comput. \textbf{24} (2015), 329--353.

\bibitem{Whittaker-Watson:1996}
Edmund~T. Whittaker and George~N. Watson, \emph{A course of modern analysis},
  Cambridge University Press, Cambridge, 1996, \lowercaseforams Reprint of the
  fourth (1927) edition.

\bibitem{Zeilberger:1990:bijection}
Doron Zeilberger, \href{http://dx.doi.org/10.1016/0012-365X(90)90047-L}{\emph{A
  bijection from ordered trees to binary trees that sends the pruning order to
  the {S}trahler number}}, Discrete Math. \textbf{82} (1990), no.~1, 89--92.
  \MR{1058712}

\end{thebibliography}

\end{document}